\documentclass[11pt,a4paper]{article}

\usepackage[colorlinks,citecolor=cyan,linkcolor=magenta]{hyperref}
\usepackage[latin1]{inputenc}
\usepackage{amsmath}
\usepackage{amsfonts}
\usepackage{amssymb}
\usepackage{theorem}
\usepackage{subcaption}
\usepackage{graphicx}
\usepackage{xspace}
\usepackage{tikz}
\usepackage{fullpage}

\newcommand{\fes}{S^{{\bf p}}_{\mathcal{T}}}

\newcommand{\coveringmesh}{{\mathcal T}_h^{\sharp}}

\newcommand{\ndg}[1]{| \kern -.25mm \|{#1}| \kern -.25mm \|}
\newcommand{\nsdg}[1]{| \kern -.25mm \|{#1}| \kern -.25mm \|_s}
\newcommand{\su}{\sum_{\k \in\mathcal{T}}}
\newcommand{\ltwo}[2]{\|{#1}\|_{#2}}

\newcommand{\el}{ \kappa \in \mathcal{T} }
\newcommand{\ud}{\mathrm{d}}

\newcommand{\norm}[2]{\|#1\|_{#2}}

\newcommand{\dint}{\text{\rm int}}
\newcommand{\mbf}[1]{\mbox{\boldmath$\rm{#1}$}}
\newcommand{\mean}[1]{ \{#1\} }
\newcommand{\jump}[1]{  [#1]  }

\newcommand{\ip}{\hat{\Pi}} 

\newcommand{\mesh}{{\mathcal T}}

\newcommand{\ddd}{{\rm D}}
\newcommand{\dn}{{\rm N}}

\DeclareMathOperator{\diam}{diam}
\DeclareMathOperator{\dist}{dist}

\newcommand{\uu}[1]{\mathbf{#1}}
\renewcommand{\k}{\kappa}

\renewcommand{\tilde}[1]{\widetilde{#1}}
\renewcommand{\hat}[1]{\widehat{#1}}
\makeatletter
\newcommand*{\rom}[1]{\text{\expandafter\@slowromancap\romannumeral #1@}}
\makeatother

\newtheorem{corollary}{Corollary}[section]
\newtheorem{lemma}[corollary]{Lemma}
\newtheorem{theorem}[corollary]{Theorem}

\newtheorem{definition}[corollary]{Definition}

\newtheorem{remark}[corollary]{Remark}

\newtheorem{assumption}[corollary]{Assumption}

\newcommand{\qed}{ \vspace{-0.5cm} \hfill $\Box$ }
\newenvironment{proof}[1][Proof.]{\begin{trivlist}
\item[\hskip \labelsep {\bfseries #1}]}{\end{trivlist}\qed}

\begin{document}
\title{Discontinuous Galerkin Methods for the Biharmonic Problem \\ on Polygonal and Polyhedral Meshes}
\author{
	 Zhaonan Dong\thanks{
		Department of Mathematics,
		University of Leicester,
		University Road,
		Leicester LE1 7RH,
		UK
		{\tt{zd14@le.ac.uk}}.
	}
}
\date{\today}

\maketitle

\begin{abstract}
\noindent

We introduce an $hp$-version symmetric interior penalty discontinuous Galerkin finite element method (DGFEM) for the numerical approximation of the biharmonic equation on general computational meshes consisting of polygonal/polyhedral (polytopic) elements. In particular, the stability and $hp$-version a-priori error bound are derived based on the specific choice of the interior penalty parameters which allows for edges/faces degeneration.  Furthermore, by deriving a new inverse inequality for a special class {of} polynomial functions (harmonic polynomials), the proposed DGFEM is proven to be stable  to incorporate very general polygonal/polyhedral  elements with an   \emph{arbitrary} number of faces for polynomial basis with degree $p=2,3$.   The key feature of the proposed method is that it employs  elemental polynomial bases of total degree $\mathcal{P}_p$, defined in the physical coordinate system, without requiring the mapping from a given reference or canonical frame. A series of numerical experiments are presented to demonstrate the performance of the proposed DGFEM on general polygonal/polyhedral meshes. 
\end{abstract}

\section{Introduction}

Fourth-order boundary-value problems have been widely used in  mathematical models from different disciplines, see \cite{gazzola2010polyharmonic}. The classical conforming finite element methods (FEMs) for the numerical solution of the biharmonic equation require that the approximate solution lie in a finite-dimensional subspace of the Sobolev space $H^2(\Omega)$. In particular, this necessitates the use of $C^1$ finite elements, such as  Argyris elements. In general, the implementation of $C^1$ elements is far from trivial.  To relax the $C^1$ continuity requirements across the element interfaces, nonconforming FEMs have been  commonly used by engineers and also analysed  by mathematicians; we refer to the monograph   \cite{ciarlet} for the details of above mentioned FEMs. For a more recent approach, we mention the $C^0$ interior penalty methods, see  \cite{MR1915664,MR2142191} for details. Another approach to avoid using $C^1$ elements is to use the mixed finite element methods, we refer to the monograph \cite{MR3097958} and the reference therein.

In the last two decades, discontinuous Galerkin FEMs (DGFEMs) have been  considerably developed  as flexible and efficient discretizations for a large class of problems ranging from computational fluid dynamics to computational mechanics and electromagnetic theory. In the pioneer work \cite{baker}, DGFEMs were first introduced as a special class of nonconforming FEMs to solve the biharmonic equation. For the overview of the historical development of  DGFEMs, we refer to the important paper \cite{unified} and monographs \cite{MR1842161,DiPietroErn} and all the reference therein.   DGFEMs are attractive as they  employ the discontinuous finite element spaces, giving great flexibility in the design of meshes and polynomial bases, providing  general framework for $hp$-adaptivity.   For the biharmonic problem, $hp$-version interior penalty (IP) DGFEMs were introduced   in \cite{MR2048235,MR2295480,MR2298696}. The stability of different IP-DGFEMs  and a priori error analysis in various norms have been  studied in those work. Additionally,  the exponential convergence for the $p$-version IP-DGFEMs were proven in \cite{MR2520159}. The a posterior error analysis of the symmetric IP-DGFEM has been done in \cite{MR2755946}. In \cite{MR2142199,MR3742892}, the domain decomposition preconditioners have been designed for IP-DGFEMs.  

More recently, DGFEMs on meshes consisting of general polygons in two dimensions or general polyhedra in three dimensions, henceforth termed collectively as \emph{polytopic}, have been proposed \cite{DGcomposite,DGpoly1,DGpoly2,DGpolyparabolic,MR2846986,MR3585793}.   The key interest of employing  polytopic meshes is predominant by the potential reduction in the total numerical degrees of freedom required for the numerical solution of PDE problems,  which is particularly important  in designing the  adaptive computations  for PDE problems on domains with micro-structures. 
Hence, polytopic meshes can naturally be combined with DGFEMs due to their element-wise discontinuous approximation.   In our works \cite{DGpoly1,DGpoly2},  an $hp$-version symmetric IP-DGFEM was introduced for the linear elliptic problem and the general advection-diffusion-reaction problem on meshes consisting of $d$-dimensional polytopic elements were analysed. The key aspect of  the method is that  the DGFEM is stable on general polytopic elements in the presence of  degenerating $(d-k)$-dimensional element facets, $k=1,\dots,d-1$, where $d$ denotes the spatial dimension. The main mesh assumption for the polytopic elements is that all the elements have a uniformly bounded number of $(d-1)$-dimensional faces, without imposing any assumptions on the measure of faces. (Assumption \ref{sec4:assumption_no_element_faces} in this work) In our work \cite{DGpolyparabolic}, we proved that  the IP-DGFEM  is stable for second order elliptic problem  on polytopic elements with arbitrary number of $(d-1)$-dimensional faces, without imposing any assumptions on the measure of faces.  The mesh assumption for the polytopic elements is that all the elements should satisfy a shape-regular condition,  without imposing any assumptions on the measure of faces or number of faces (Assumption \ref{sec5:new_assumption_no_element_faces} in this work). For details of  DGFEMs on polytopic elements, we refer to the monograph \cite{DGpolybook}. 

 To support such  general element shapes, without destroying  the local approximation properties of the DGFEM developed in \cite{DGpoly1,DGpoly2,DGpolyparabolic}, polynomial spaces defined in the physical frame, rather than mapped polynomials from a reference element, are typically employed. 
 It has been demonstrated numerically that the DGFEM employing $\mathcal{P}_p$-type basis achieves a faster rate of convergence, with respect to the number of degrees of freedom present in the underlying finite element space, as the polynomial degree $p$ increases, for a given fixed mesh, than the respective DGFEM employing a (mapped) $\mathcal{Q}_p$  basis on tensor-product elements; we refer  \cite{Peter_phd} for more numerical examples.  The proof of the above numerical observations is given in  \cite{D17}.

In this work,  we will extend the results in \cite{DGpoly1,DGpoly2,DGpolyparabolic} to cover $hp$-version IP-DGFEMs for biharmonic PDE problems. We will prove the stability and derive the a priori error bound for the $hp$-version IP-DGFEM on general polytopic elements with possibly degenerating $(d-k)$-dimensional facets,  under two different  mesh assumptions. (Assumption \ref{sec4:assumption_no_element_faces} and \ref{sec5:new_assumption_no_element_faces}). The key technical difficulty is that the $H^1$-seminorm to $L_2$-norm inverse inequality for general polynomial functions defined on polytopic elements with arbitrary number of faces  is empty in the literature. To address this issue, we prove a new inverse inequality for \emph{harmonic polynomial functions}  on polytopic elements satisfying Assumption \ref{sec5:new_assumption_no_element_faces}. With the help of the new inverse inequality,  we  prove the stability and derive the a priori error bound for the proposed DGFEM employing  $\mathcal{P}_p$ basis, $p=2,3$, under the Assumption \ref{sec5:new_assumption_no_element_faces}.  Here, we mention that there already exist different polygonal discretization methods for biharmonic problems \cite{MR3002804,MR3529253,MR3741104,MR3190348}.  To the best of {the} author's understanding, the proposed DGFEM for biharmonic problem is the first polygonal discretization scheme whose stability and approximation are independent of the relative size of  elemental faces compared to the element diameter,  and even independent of the number of elemental faces, for $\mathcal{P}_p$ basis with  $p=2,3$.  {We point out that the  $\mathcal{P}_p$ basis with  $p=2,3$  satisfies the condition that the Laplacian of any function is a harmonic  polynomials.}

The remainder of this work is structured as follows. In Section \ref{Problem}, we introduce the model problem and define the finite element space. In Section \ref{sec:IPDG}, the $hp$-version symmetric interior penalty  discontinuous Galerkin finite element method is introduced. In Section \ref{bounded_faces},  we present the stability analysis and a priori error analysis for the proposed DGFEM over polytopic meshes with bounded number of element faces. In Section \ref{unbounded_faces}, we will derive the new inverse inequality for polytopic meshes with arbitrary number of elemental faces satisfying Assumption \ref{sec5:new_assumption_no_element_faces}. Then, we present the stability analysis and error analysis. A series of numerical examples are presented in Section \ref{numerical example}. Finally, we make {concluding} remarks in Section  \ref{conclusion}.

\section{Problem and Method}\label{Problem}

For a Lipschitz domain $\omega \subset {\mathbb R}^d$, $d = 2,3$, 
we denote by $H^s(\omega)$ the Hilbertian Sobolev space of  index $s\ge 0$ of real--valued functions defined on
$\omega$, endowed with  seminorm $|\cdot |_{H^s(\omega)}$ and norm $\|\cdot\|_{H^s(\omega)}$. Furthermore, we let~$L_p(\omega)$, $p\in[1,\infty]$, 
be the standard Lebesgue space on $\omega$, equipped with the norm~$\|\cdot\|_{L_p(\omega)}$. Finally,  $|\omega|$ denotes the $d$--dimensional Hausdorff measure of $\omega$.

\subsection{Model problem}\label{model}

Let $\Omega$ be a bounded open polyhedral domain in $\mathbb{R}^d$, $d=2,3$. We consider the biharmonic equation
\begin{equation}\label{pde}
\Delta^2  u =f \quad\text{in } \Omega,
\end{equation}
where $f\in L_2(\Omega)$. We impose Dirichlet boundary conditions
\begin{equation}\label{bcs}
\begin{aligned}
u =&\ g_{\ddd}^{},\qquad \text{on } \partial \Omega,\\
\nabla u\cdot \mbf{n} = &\ g_{\dn}^{}, \qquad \text{on }  \partial \Omega,
\end{aligned}
\end{equation}
where $\mbf{n}$ denotes the unit outward normal vector to $\partial \Omega$.  It is well-known that by choosing $g_{\ddd}^{} \in H^{3/2}(\partial \Omega)$, $g_{\dn}^{} \in H^{1/2}(\partial \Omega)$, {the problem \eqref{pde} is well-posed with $u\in H^2(\Omega)$} (see \cite[page 15]{MR851383}).

\subsection{Finite element spaces}

We shall adapt the setting of meshes from \cite{DGpolybook}.  Let $\mesh$ be a 
subdivision of the computational domain $\Omega$ into 
disjoint open polygonal $(d=2)$ or polyhedral $(d=3)$ elements $\k$ such that $\bar{\Omega}=\cup_{\el}\bar{\k}$    and denote by  $h_{\k}$ the diameter of  $\el$; i.e., $h_{\k}:=\diam(\k)$. 
In the absence of hanging nodes/edges, we define the {\em interfaces} of the mesh $\mesh$ to be the set of 
$(d-1)$--dimensional facets of the elements $\k \in \mesh$. To facilitate the presence of hanging nodes/edges, which are permitted in $\mesh$,
the interfaces of $\mesh$ are defined to be the intersection of the $(d-1)$--dimensional facets of neighbouring elements. {In the case when
$d=2$, the interfaces of $ \mesh$ are simply piecewise linear segments ($(d-1)$--dimensional simplices). However, in general for $d=3$, the interfaces of $\mesh$ consist of general polygonal surfaces in $\mathbb{R}^3$. Thereby, we assume that each planar section of each interface of an element $\k \in \mesh$ may be subdivided into a set of co-planar triangles ($(d-1)$--dimensional
simplices). }


As in \cite{DGpoly1,DGpoly2}, we assume that a sub-triangulation into faces of each mesh interface is given if $d=3$, and denote by $\mathcal{E}$ the union of all open mesh interfaces if $d=2$ and the union of all open triangles belonging to the sub-triangulation of all mesh interfaces if $d=3$. In this way, $\mathcal{E}$ is always defined as a set of $(d-1)$--dimensional simplices. 
 Further, 
we write $\mathcal{E}_{\dint}$ and $\mathcal{E}_{\ddd}$ to denote the union of all open $(d-1)$--dimensional element faces $F\subset \mathcal{E}$ that are contained in $\Omega$ and in $\partial\Omega$, respectively. Let $\Gamma_{\dint}:=\{\uu{x} \in \Omega:\uu{x} \in F , F\in \mathcal{E}_{\dint} \}$ and let $\Gamma_{\ddd}:=\{\uu{x}\in \partial \Omega:\uu{x} \in F , F\in \mathcal{E}_{\ddd} \}$, while $\Gamma: = \Gamma_{\dint} \cup \Gamma_{\ddd}$.  

Given $\el$, we write $p_{\k}$ to denote the (positive) \emph{polynomial degree} of the element $\k$, and collect the $p_{\k}$ in
the vector ${\bf p}:=(p_{\k}:\el)$. We then define the \emph{finite element space} $\fes$ with
respect to $\mesh$ and ${\bf p}$ by
\[
\fes:=\{u\in L_2(\Omega)
:u|_{\k}\in\mathcal{P}_{p_{\k}}(\k),\el\},
\]
where
$\mathcal{P}_{p_\k}(\k)$ denotes the space of polynomials of total degree $p_\k$ on $\k$, satisfying $p_\k \geq 2$ for all $\k \in \fes$. As in \cite{DGpoly1}, we point out that the local elemental polynomial spaces employed within the definition
of $\fes$ are defined in the physical coordinate system, without the need to map from a given reference or canonical frame. 
Finally, we define the  broken Sobolev space $H^\bold{s}(\Omega,\mesh)$  with respect to the subdivision  $\mesh$ up to composite order $\bold{s}$ as follows
\begin{equation}\label{brokenSob}
H^\bold{s} (\Omega,\mesh)=\{ u\in L_2(\Omega):u|_\k \in H^{s_\k}(\k) \quad \forall \k \in \mesh \},
\end{equation}
which will be used to construct the forthcoming DGFEM.

\subsection{Trace operators}
For any element $\k\in \mesh$, we denote by $\partial \k$ the union of $(d-1)$-dimensional open faces of $\k$. Let $\k_i$ and $\k_j$ be two 
adjacent elements of ${\cal T}$ and let $\uu{x}$ be an arbitrary point on the interior face $F \subset \Gamma_{\dint}$ given
by $F = \partial \k_i \cap \partial \k_j$. We write $\mbf{n}_{i}$ and $\mbf{n}_{j}$  to denote the 
outward unit normal vectors on $F$, relative to $\partial\k_i$ and $\partial\k_j$, respectively. 
Furthermore, let $v$ and $\mbf{q}$  be  scalar- and vector-valued functions,  which are  smooth inside
each element~$\k_i$ and $\k_j$. By $(v_i,\mbf{q}_i)$ and $(v_j,\mbf{q}_j)$, we denote the traces of
$(v,\mbf{q})$ on $F$ taken from within the interior of $\k_i$ and $\k_j$,
respectively.   The averages of $v$ and $\mbf{q}$  at $\uu{x}\in
F$ are given by
$$
\mean{v}:=\frac{1}{2}(v_i+v_j), \quad \mean{\mbf{q}}:=\frac{1}{2}(\mbf{q}_i+\mbf{q}_j),$$
respectively. Similarly, the jump of $v$ and $\mbf{q}$ at
$\uu{x}\in F \subset \Gamma_{\text{int}}$ are given by
$$
\jump{v} :=v_i\,{\mbf{n}}_{i}+v_j\,{\mbf{n}}_{j}, \quad \jump{\mbf{q}} := \mbf{q}_i\cdot {\mbf{n}}_{i} +\mbf{q}_j\cdot {\mbf{n}}_{j},
$$
respectively.
On a boundary face $F\subset  \Gamma_{\ddd}$, such that $F \subset \partial \k_i$, $\k_i\in {\cal T}$,
we set 
$$
\mean{v}=v_i, \quad    \mean{\mbf{q}}=\mbf{q}_i,  \quad \jump{v}=v_i {\mbf{n}}_{i}  \quad \jump{\mbf{q}}=\mbf{q}_i \cdot {\mbf{n}}_{i},
$$ 
with~$\mbf{n}_{i}$ denoting the unit outward normal
vector on the boundary $\Gamma_{\ddd}$.

\section{Interior Penalty Discontinuous Galerkin Method} \label{sec:IPDG}

With help of the above notation, we can now introduce the DGFEM for the problem \eqref{pde}, \eqref{bcs}: Find $u_h\in\fes$ such that 
\begin{equation}\label{galerkin_dg}
B(u_h,v_h)=\ell(v_h) \quad \text{for all } v_h \in \fes, 
\end{equation}
where the bilinear form $B(\cdot, \cdot ):\fes\times \fes\to \mathbb{R}$ is defined {by}
\begin{equation} \label{bilinear-form}
\begin{aligned}
B(u,v) := &  \su \int_\k  \Delta u \Delta v \ud \uu{x} 
+\int_\Gamma \Big(  \mean{\nabla  \Delta u} \cdot \jump{v}
+ \mean{\nabla \Delta v} \cdot \jump{u}\Big)  \ud{s}  \\ 
-&  \int_\Gamma \Big( 
  \mean{\Delta u}  \jump{\nabla v}
+ \mean{\Delta v}  \jump{\nabla u} \Big) \ud{s}
+\int_\Gamma \Big( \sigma \jump{u} \cdot \jump{v}
+\tau \jump{\nabla u} \jump{\nabla v} \Big) \ud{s}.\end{aligned}
\end{equation}
Furthermore,  the linear functional $\ell:\fes\to\mathbb{R}$ is defined by
\begin{equation} \label{linear-form}
\begin{aligned}
\ell(v)
&:=& \su\int_{\k} f  v \ud \uu{x}
+ \int_{\Gamma_{\ddd}} g_{\ddd} \Big( \nabla \Delta v \cdot \mbf{n} + \sigma v \Big) 
+ g_{\dn}  \Big(  \tau \nabla v \cdot \mbf{n} -\Delta v \Big)\ud s. 
\end{aligned}
\end{equation}
The well-posedness and stability properties of the above method depend on the choice of the discontinuity-penalization functions $\sigma \in L_\infty(\Gamma)$ and $\tau \in L_\infty(\Gamma)$  appearing in~\eqref{bilinear-form} and~\eqref{linear-form} .  The precise definition will be given in next sections based on employing different mesh assumptions on the elements present in the computational mesh $\mesh$. 

\begin{remark}
The DGFEM formulation introduced in this work coincides with the SIP-DGFEM defined in \cite{MR2298696,MR2295480,MR2520159,MR2755946,MR3742892}, which contains {the  inner product of the Laplacian of functions. We note that for the alternative formulation of the biharmonic problem based on Frobenius} product of the Hessians of functions in the literature, see \cite{MR2670114,MR2142191}, the forthcoming analysis and results in this work are also valid.
\end{remark}


\section{Error Analysis I: Bounded Number of Element Faces} \label{bounded_faces}
In this section, we study the stability and a priori error analysis of the DGFEM \eqref{galerkin_dg} under the following mesh assumption, which guarantees that the number of faces each element possesses remains bounded under mesh refinement. 

\begin{assumption}[Limited number of faces]  \label{sec4:assumption_no_element_faces} 
For each element $\el$, we define
$$
C_\k =  \mbox{card} 
      \Big\{ F \in \Gamma : 
              F \subset \partial \k \Big\} . 
$$
We assume there exists a positive constant $C_F$, 
independent of the mesh parameters, such that
$$
\max_{\el} C_\k 
\leq C_F. 
$$
\end{assumption}

\subsection{Inverse estimates} \label{sec4:inverse}
In this section, we will revisit the  inverse inequalities in  the context of general polytopic elements from \cite{DGpoly1,DGpoly2} without proof. The detail of proof can be found in Chapter 3 of \cite{DGpolybook}.
To this end, we introduce the following set of definitions and mesh assumptions.

\begin{definition}\label{sec4:element_face_simplices}
For each element $\k$ in the computational mesh $\mesh$, we define the 
family $\mathcal{F}_{\flat}^{\k}$ of all possible $d$--dimensional simplices 
contained in $\k$ and having at least one face in common with $\k$.  
Moreover, we write $\k_{\flat}^F$ to denote a simplex belonging to  
$\mathcal{F}_{\flat}^{\k}$ which shares with $\el$ the specific face 
$F\subset\partial\k$.
\end{definition}

\begin{definition}\label{sec4:def_poly_assumption}
An element $\k\in\mesh$ is said to be {\em $p$-coverable} with respect to $p\in\mathbb{N}$, if there exists a set of $m_{\k}$ overlapping shape-regular simplices $K_i$, $i=1,\dots, m_{\k}$, $m_{\k}\in\mathbb{N}$,
such that 
\begin{eqnarray}
\dist(\k, \partial K_i) <  C_{as}\frac{\diam(K_i)}{p^{2}},
\qquad
\mbox{and} 
\qquad
|K_i|\ge c_{as} |\k| \label{sec4:T_tilde_condition}
\end{eqnarray}
for all $i=1,\dots, m_{\k}$, where $C_{as}$ and $c_{as}$ are
positive constants, independent of $\k$ and $\mesh$.
\end{definition}

Equipped with Definition~\ref{sec4:def_poly_assumption}, we are now 
in a position to present the following $hp$--version trace inverse inequality for 
general polytopic elements which directly accounts for elemental facet degeneration.

\begin{lemma}\label{sec4:lemmal_inv}
Let $\el$, $F\subset \partial \k$ denote one of its faces.
Then, for each $v\in\mathcal{P}_p(\k)$, the following inverse inequality holds
\begin{equation}\label{sec4:inv_est_gen}
\norm{v}{L_2(F)}^2 \le C_{\rm INV}(p,\k,F) 
p^2 \frac{|F|}{|\k|}\norm{v}{L_2(\k)}^2,
\end{equation}
where
\begin{eqnarray}
C_{\rm INV}(p,\k,F)
: = \left\{
\begin{array}{ll}
\displaystyle{C_{{\rm inv},1} \min \Big\{\frac{|\k|}{\sup_{\k_{\flat}^F\subset \k}
|\k_{\flat}^F|},p^{2(d-1)} \Big\}}, ~~ & \text{if $\k$ is $p$-coverable} 
\\ 
\displaystyle{C_{{\rm inv},1} \frac{|\k|}{\sup_{\k_{\flat}^F\subset \k}
|\k_{\flat}^F|}}, &  \text{otherwise,} 
\end{array}  
\right. \label{sec4:inverse_constant}
\end{eqnarray}
and with $\k_{\flat}^F\in \mathcal{F}_{\flat}^{\k}$ as in Definition~\ref{sec4:element_face_simplices}.
Furthermore, $C_{{\rm inv},1}$ are positive constants which are
independent of $|\k|/\sup_{\k_{\flat}^F\subset \k}
|\k_{\flat}^F|$, $|F|$, $p$, and $v$. 
\end{lemma}
Next, in order to present an inverse inequality which provides a bound on the
$H^1(\k)$--seminorm of a polynomial function $v$, $\el$, with respect to the 
$L_2(\k)$--norm of $v$, on the general polytopic meshes $\k \in \mesh$. It is now necessary to assume shape-regularity of the polytopic mesh $\k \in\mesh$.

\begin{assumption} \label{sec4:assumption_shape_regular_mesh}
The subdivision $\mesh$ is shape-regular, in the sense of \cite{ciarlet}, i.e., there exists a positive constant 
$C_{\rm r}$, independent of the mesh parameters, such that
\begin{equation*}
\forall \k \in\mesh, \quad  \frac{h_\k}{\rho_\k} \leq C_{\rm r}.
\end{equation*}
with $\rho_\k$ denoting the diameter of the largest ball contained in $\k$.
\end{assumption}

\begin{lemma}\label{sec4:DG-lemma inverse}
Given that Assumption~\ref{sec4:assumption_shape_regular_mesh} is
satisfied, then, for any $\k\in\mesh$ which is $p$-coverable
and $v\in\mathcal{P}_p(\k)$, the following inverse inequality holds
\begin{equation} \label{sec4:inverse_h1}
\|\nabla v\|_{L_2(\k)}^2  \le  C_{{\rm inv},2} \frac{p^4}{h_{\k}^2}\norm{v}{L_2(\k)}^2,
\end{equation}
where $C_{{\rm  inv},2}$ is a positive constant, which is independent of $v$,
$h_\k$ and $p$, 
but depends on the shape-regularity constant of the covering of $\k$.
\end{lemma}

\begin{remark}
We emphasize  that the above inverse inequalities in Lemma \ref{sec4:lemmal_inv},  \ref{sec4:DG-lemma inverse} are both sharp with respect to $(d-k)$-dimensional faces degeneration, for $k=1,\dots, d-1$. Moreover, they are both essential for proving stability of DGFEM \eqref{galerkin_dg} over general polytopic elements $\k\in \mesh$ with degenerating faces. 
\end{remark}

\subsection{The stability of DGFEM} \label{sec4:stability}

For the forthcoming error analysis, we introduce an inconsistency formulation of the bilinear form \eqref{bilinear-form} and linear form  \eqref{linear-form}, without using polynomial lifting operators, cf. \cite{MR2520159}. We define, for $u,v\in \mathcal{S}:= H^2(\Omega)+\fes$,  the bilinear form 
\begin{equation} \label{in-bilinear-form}
\begin{aligned}
\tilde{B}(u,v) := &  \su \int_\k  \Delta u \Delta v \ud \uu{x} 
+\int_\Gamma \Big(  \mean{\nabla \ip (\Delta u)} \cdot \jump{v}
+ \mean{\nabla \ip (\Delta v)} \cdot \jump{u}\Big)  \ud{s}  \\ 
-&  \int_\Gamma \Big( 
  \mean{\ip(\Delta u)}  \jump{\nabla v}
+ \mean{\ip (\Delta v)}  \jump{\nabla u} \Big) \ud{s}
+\int_\Gamma \Big( \sigma \jump{u} \cdot \jump{v}
+\tau \jump{\nabla u} \jump{\nabla v} \Big) \ud{s}.
\end{aligned}
\end{equation}
 and the linear functional $\tilde{\ell}:\mathcal{S}\to\mathbb{R}$ by
\begin{equation} \label{in-linear-form}
\begin{aligned}
\tilde{\ell}(v)
&:=& \su\int_{\k} f  v \ud \uu{x}
+ \int_{\Gamma_{\ddd}} g_{\ddd} \Big( \nabla \ip(\Delta v) \cdot \mbf{n} + \sigma v \Big) 
+ g_{\dn}  \Big(  \tau \nabla v \cdot \mbf{n} -\ip(\Delta v) \Big)\ud s;
\end{aligned}
\end{equation}
here, $\ip : L_2(\Omega)\rightarrow S_{\mesh}^{{\bf{p}}-2}$ denotes the $L_2$-projection onto the finite element space $S_{\mesh}^{{\bf{p}}-2}$. It is immediately clear, therefore, {that} $\tilde{B}(u_h,v_h)={B}(u_h,v_h)$ and $\tilde{\ell}(v_h)={\ell}(v_h)$ for all $u_h, v_h \in \fes$.  

Next, we introduce the DGFEM-norm $\ndg{\cdot}$: 
\begin{equation} \label{DG_norm}
\ndg{v}^2: = \sum_{\el} \ltwo{\Delta v}{L_2(\k)} ^2 +\int_{\Gamma} \Big( \sigma |\jump{v}|^2 +\tau |\jump{\nabla v}|^2 \Big)\ud{s},
\end{equation}
for all functions $v\in H^2 (\Omega,\mesh)$.
The continuity and coercivity of the inconsistent bilinear form $\tilde{B}(\cdot,\cdot)$, with respect to the norm $\ndg{\cdot}$, is established by the following lemma.

\begin{lemma} \label{sec4:lem:coercivity}
Given that Assumption~\ref{sec4:assumption_no_element_faces}, \ref{sec4:assumption_shape_regular_mesh} hold, and let 
{$\sigma: \Gamma\rightarrow \mathbb{R_+}$  and $\tau: \Gamma\rightarrow \mathbb{R_+}$  be} defined facewise: 
\begin{equation}\label{sec4:eq:sigma}
\sigma(\mbf{x}) :=\left\{
\begin{array}{ll}
C_{\sigma}{\displaystyle   \max_{\k\in\{\k_i,\k_j\}}\Big\{ (C_{\rm INV}(p_\k,\k,F)\frac{ p_\k^2|F|}{|\k|} )
(C_{{\rm inv},2} \frac{p^4_\k}{h_{\k}^2}) \Big\} }, & ~  \uu{x}\in F \in \Gamma_{\dint}, ~F\subset\partial\k_i\cap\partial\k_j, \\
C_{\sigma} (C_{\rm INV}(p_\k,\k,F)
~\displaystyle\frac{ p_\k^2|F|}{|\k|})(C_{{\rm inv},2} \frac{p^4_\k}{h_{\k}^2}) , & ~  \uu{x}\in F \in \Gamma_{\ddd}, ~F\subset\partial\k.
\end{array}
\right. 
\end{equation}
and 
\begin{equation}\label{sec4:eq:tau}
\tau(\mbf{x}) :=\left\{
\begin{array}{ll}
C_{\tau}{\displaystyle   \max_{\k\in\{\k_i,\k_j\}}\Big\{ C_{\rm INV}(p_\k,\k,F)
~\frac{ p_\k^2|F|}{|\k|} \Big\} }, & ~  \uu{x}\in F \in \Gamma_{\dint}, ~F\subset\partial\k_i\cap\partial\k_j, \\
C_{\tau} C_{\rm INV}(p_\k,\k,F)
~\displaystyle\frac{ p_\k^2|F|}{|\k|}, & ~  \uu{x}\in F \in \Gamma_{\ddd}, ~F\subset\partial\k.
\end{array}
\right.
\end{equation}
Where $C_\sigma$ and  $C_\tau$ are  sufficiently large positive
constants. Then the bilinear form $\tilde{B}(\cdot,\cdot)$ is coercive and continuous over $\mathcal{S}\times \mathcal{S}$, i.e.,
\begin{equation}\label{sec4:coer}
\tilde{B}(v,v) \geq C_{\rm {coer}} \ndg{v}^2 \quad\text{for all}\quad v\in \mathcal{S},
\end{equation}
and
\begin{equation}\label{sec4:cont}
\tilde{B}(w,v)\leq C_{\rm cont}\ndg{w}\,\ndg{v} \quad\text{for all}\quad w,v\in \mathcal{S},
\end{equation}
respectively,
where $C_{\rm coer}$ and $C_{\rm cont}$ are positive constants, independent of the local
mesh sizes $h_\k$, local polynomial degree orders ${p_\k}$, $\el$ and measure of faces.
\end{lemma}

\begin{proof}
The proof is based on employing standard arguments. {Firstly, we will prove \eqref{sec4:coer}.  For any} $v\in \mathcal{S}$, we have the following identity
\begin{equation}\label{sec4:coercivity-1}
\tilde{B}(v,v)
=\ndg{v}^2 + 2\int_{\Gamma}\mean{\nabla \ip (\Delta v)} \cdot \jump{v} \ud s
- 2\int_{\Gamma}\mean{ \ip (\Delta v)} \jump{\nabla v}\ud s.
\end{equation}
We start to bound the second term on the right-hand side of \eqref{sec4:coercivity-1}. To this end, given $F\in \Gamma_{\dint}$, such that $F\subset \partial \k_i \cap \partial \k_j$, with $\k_i,\k_j \in \mesh$, upon employing   the Cauchy--Schwarz inequality and 
the arithmetic--geometric mean inequality, we deduce that
\begin{equation*} 
\begin{aligned}
 \int_{F}\mean{\nabla \ip (\Delta v)} \cdot \jump{v} \ud s 
&\leq
 \frac{1}{2}\left(\norm{{\sqrt{\sigma^{-1}}}  \nabla \ip(\Delta v_i)}{L_2(F)}+\norm{{\sqrt{\sigma^{-1}}}  \nabla \ip(\Delta v_j)}{L_2(F)}  \right)          
 \norm{\sqrt{\sigma}\jump{v}}{L_2(F)} \\
 &\leq 
\epsilon \left(\norm{{\sqrt{\sigma^{-1}}}  \nabla \ip(\Delta v_i)}{L_2(F)}^2+\norm{{\sqrt{\sigma^{-1}}} \nabla  \ip(\Delta v_j)}{L_2(F)}^2  \right)   
 + \frac{1}{8\epsilon} \norm{\sqrt{\sigma}\jump{v}}{L_2(F)}^2. 
\end{aligned}
\end{equation*}
Using the inverse inequalities stated in Lemma~\ref{sec4:lemmal_inv} and Lemma~\ref{sec4:DG-lemma inverse}, we deduce {that }
\begin{equation} 
\begin{aligned}
& \int_{F}\mean{\nabla \ip (\Delta v)} \cdot \jump{v} \ud s  \\
&\leq
\epsilon  \left(
C_{\rm INV}(p_{\k_i},{\k_i},F)\frac{ p_{\k_i}^2|F|}{|{\k_i}|}  \norm{{\sqrt{\sigma^{-1}}}  \nabla \ip(\Delta v_i)}{L_2(\k_i)}^2 \right. \\
& \quad 
\left.+ 
C_{\rm INV}(p_{\k_j},{\k_j},F)\frac{ p_{\k_j}^2|F|}{|{\k_j}|}
  \norm{{\sqrt{\sigma^{-1}}}  \nabla \ip(\Delta v_j)}{L_2(\k_j)}^2
 \right)   
 + \frac{1}{8\epsilon} \norm{\sqrt{\sigma}\jump{v}}{L_2(F)}^2 \\
 & \leq \epsilon  \left(
C_{\rm INV}(p_{\k_i},{\k_i},F)\frac{ p_{\k_i}^2|F|}{|{\k_i}|}  
(C_{{\rm inv},2} \frac{p^4_{\k_i}}{h_{\k_i}^2}) \sigma^{-1}
\norm{  \ip(\Delta v_i)}{L_2(\k_i)}^2 \right. \\
& \quad 
\left.+ 
C_{\rm INV}(p_{\k_j},{\k_j},F)\frac{ p_{\k_j}^2|F|}{|{\k_j}|}
(C_{{\rm inv},2} \frac{p^4_{\k_j}}{h_{\k_j}^2}) \sigma^{-1}
\norm{  \ip(\Delta v_j)}{L_2(\k_j)}^2
 \right)   
 + \frac{1}{8\epsilon} \norm{\sqrt{\sigma}\jump{v}}{L_2(F)}^2.
\end{aligned}
\end{equation}
By using the stability of the $L_2$-projector $\ip$ in the $L_2$-norm together with the Definition of $\sigma$ in \eqref{sec4:eq:sigma}, we have 
\begin{equation} \label{coer-relation1}
\begin{aligned}
 \int_{F}\mean{\nabla \ip (\Delta v)} \cdot \jump{v} \ud s  
 & \leq \frac{\epsilon}{C_\sigma}  \left(
\norm{  \Delta v_i}{L_2(\k_i)}^2 
+ 
\norm{ \Delta v_j}{L_2(\k_j)}^2
 \right)   
 + \frac{1}{8\epsilon} \norm{\sqrt{\sigma}\jump{v}}{L_2(F)}^2.
\end{aligned}
\end{equation}
Similarly, for $F\in \Gamma_{\ddd}$, where $F \subset \partial \k \cap \partial \Omega$, $\k\in \mesh$, we have 
\begin{equation} \label{coer-relation2}
\begin{aligned}
 \int_{F}\mean{\nabla \ip (\Delta v)} \cdot \jump{v} \ud s  
 & \leq \frac{\epsilon}{C_\sigma}  
\norm{  \Delta v}{L_2(\k)}^2 
 + 
 \frac{1}{4\epsilon} \norm{\sqrt{\sigma}\jump{v}}{L_2(F)}^2.
\end{aligned}
\end{equation}
Next, we {seek a bound on} the last term on the right-hand side of \eqref{sec4:coercivity-1}. We point out that only using the trace inverse inequality in Lemma~\ref{sec4:lemmal_inv} and Definition of $\tau$ in \eqref{sec4:eq:tau}, the following two relations hold.  For $F\in \Gamma_{\dint}$, where $F \subset \partial \k_i \cap \partial \k_j$, we have 
\begin{equation} \label{coer-relation3}
\begin{aligned}
 \int_{F}\mean{ \ip (\Delta v)} \jump{\nabla v} \ud s  
 & \leq \frac{\epsilon}{C_\tau}  \left(
\norm{  \Delta v_i}{L_2(\k_i)}^2 
+ 
\norm{ \Delta v_j}{L_2(\k_j)}^2
 \right)   
 + \frac{1}{8\epsilon} \norm{\sqrt{\tau}\jump{\nabla v}}{L_2(F)}^2.
\end{aligned}
\end{equation}
Similarly, for $F\in \Gamma_{\ddd}$, where $F \subset \partial \k \cap \partial \Omega$, $\k\in \mesh$, we have 
\begin{equation} \label{coer-relation4}
\begin{aligned}
 \int_{F}\mean{ \ip (\Delta v)}  \jump{\nabla v} \ud s  
 & \leq \frac{\epsilon}{C_\tau}  
\norm{  \Delta v}{L_2(\k)}^2 
 + 
 \frac{1}{4\epsilon} \norm{\sqrt{\tau}\jump{\nabla v}}{L_2(F)}^2.
\end{aligned}
\end{equation}
Inserting \eqref{coer-relation1}, \eqref{coer-relation2}, \eqref{coer-relation3} and \eqref{coer-relation4} into the \eqref{sec4:coercivity-1}, we deduce that 
$$
\tilde{B}(v,v)
\geq 
\Big(1 - \frac{2\epsilon C_F}{C_\sigma}  -  \frac{2\epsilon C_F}{C_\tau}\Big) 
\sum_{\el} \ltwo{\Delta v}{L_2(\k)} ^2 
+
\Big(1 -\frac{1}{2\epsilon} \Big) 
\int_{\Gamma}  \Big( \sigma |\jump{v}|^2 
+
\tau |\jump{\nabla v}|^2 \Big)\ud{s},
$$
because the number of element faces {is} bounded by Assumption \ref{sec4:assumption_no_element_faces}. So the bilinear form $\tilde{B}(\cdot,\cdot)$ is coercive over $\mathcal{S} \times \mathcal{S}$, if $C_\sigma > 4\epsilon C_F$, $C_\tau > 4\epsilon C_F$ and $\epsilon>\frac{1}{2}$.  The proof of continuity follows immediately.  

\end{proof}


\subsection{Polynomial approximation} \label{sec4:approximation}
In this section, we will revisit the polynomial approximation results in  the context of general polytopic elements from \cite{DGpolybook} without proof. To this end, we introduce the definition and  the assumption  of suitable covering of the mesh by an overlapping set of shape-regular simplices.

\begin{definition}\label{sec4:definition_mesh_covering}
We define the \emph{covering} $\coveringmesh = \{ \mathcal{K} \}$ 
related to the computational mesh $\mesh$ as a   
set of open shape-regular $d$--simplices  $\mathcal{K}$, such that, for each $\el$, 
there exists a $\mathcal{K}\in\coveringmesh$, such that $\k\subset\mathcal{K}$. 
Given $\coveringmesh$, we denote by $\Omega_{\sharp}$ the \emph{covering domain} given by $\bar{\Omega}_{\sharp}:=\cup_{\mathcal{K}\in\coveringmesh}\bar{\mathcal{K}}$. 
\end{definition}

\begin{assumption}\label{sec4:assumption_overlap}
We assume that there exists a covering $\coveringmesh$ of $\mesh$ and a positive constant 
${\cal O}_{\Omega}$, independent of the mesh parameters, such that
$$
\max_{\k \in \mesh} \mbox{card} 
      \Big\{ \k^\prime \in \mesh : 
              \k^\prime \cap \mathcal{K} \neq \emptyset, ~\mathcal{K}\in\coveringmesh ~\mbox{ such that } ~\k\subset\mathcal{K} \Big\} \leq {\cal O}_{\Omega},
$$
and
$$
h_{\mathcal{K}}:=\diam(\mathcal{K})\le C_{\diam} h_{\k},
$$ 
for each pair $\el$, $\mathcal{K}\in\coveringmesh$, with $\k\subset\mathcal{K}$, for a constant $C_{\diam}>0$, uniformly with respect to the mesh size.
\end{assumption}

 We point out that functions defined in 
$\Omega$ can be extended to the covering domain $\Omega_{\sharp}$ using the following classical extension operator, cf. \cite{stein}.

\begin{theorem} \label{sec4:thm-extension}
Let $\Omega$ be a domain with  Lipschitz boundary. Then there exists a linear extension
operator $\frak{E}:H^s(\Omega) {\rightarrow} H^s({\mathbb R}^d)$, $s \in {\mathbb N}_0$, such that
$\frak{E}v|_{\Omega}=v$ and
$$
\| \frak{E} v \|_{H^s({\mathbb R}^d)} \leq C_{\frak{E}} \| v \|_{H^s(\Omega)},
$$
where $C_{\frak{E}}$ is a positive constant depending only on $s$ and $\Omega$.
\end{theorem}

\begin{lemma}\label{sec4:lemma_approx_polytopes}
Let $\el$, $F \subset \partial \k$ denote one of its faces, and 
$\mathcal{K}\in\coveringmesh$ be the corresponding simplex, 
such that $\k\subset\mathcal{K}$, cf. Definition~\ref{sec4:definition_mesh_covering}.
Suppose that $v\in L_2(\Omega)$ is such that 
$\frak{E}v|_{\mathcal{K}}\in H^{l_{\k}}(\mathcal{K})$, for some $l_\k\ge 0$. Then, given
Assumption \ref{sec4:assumption_overlap} is satisfied, there {exists}  $\tilde{\Pi}_p v|_\k \in \mathcal{P}_p(\k)$, such that following bounds hold
\begin{equation}\label{sec4:approxH_k}
\norm{v - \tilde{\Pi}_p v}{H^q(\k)} 
\le C_{1} \frac{h_{\k}^{s_{\k}-q}}{p^{l_{\k}-q}}\norm{\frak{E}v}{H^{l_{\k}}(\mathcal{K})},\quad l_\k\ge 0,
\end{equation}
for $0\le q\le l_\k$,  
\begin{equation}\label{sec4:approx-inf_k}
\norm{v - \tilde{\Pi}_p v}{L_2(F)}
\le C_{2} |F|^{1/2}\frac{h_{\k}^{s_{\k}-d/2}}{p^{l_{\k}-1/2}} 
C_m(p,\k,F)^{1/2}
\norm{\frak{E} v}{H^{l_\k}(\mathcal{K})}, ~ l_{\k}>d/2,~~~~~~~~
\end{equation}
{and 
\begin{equation}\label{sec4:approx-inf_k_grad}
\norm{\nabla(v - \tilde{\Pi}_p v)}{L_2(F)}
\le C_{3} |F|^{1/2}\frac{h_{\k}^{s_{\k}-(d+2)/2}}{p^{l_{\k}-3/2}} 
C_m(p,\k,F)^{1/2}
\norm{\frak{E} v}{H^{l_\k}(\mathcal{K})}, ~ l_{\k}>(d+2)/2,~~~~~~~~
\end{equation} }
where
$$
C_m(p,\k,F) 
= \min \Big\{ \frac{h_\k^d}{\sup_{\k_{\flat}^F\subset \k}|\k_{\flat}^F|},p^{d-1} \Big\},
$$
$s_{\k}=\min\{p+1, l_{\k}\}$ and {$C_{1}$, $C_{2}$  and $C_{3}$
are positive constants, }
which depend on the shape-regularity of ${\mathcal{K}}$, but are
independent of $v$, $h_{\k}$, and $p$.
\end{lemma}
\begin{proof}
{The proof for  relation \eqref{sec4:approxH_k} and \eqref{sec4:approx-inf_k} can be found  in \cite[Lemma $23$]{DGpolybook}.  The relation \eqref{sec4:approx-inf_k_grad} can be proven in the similar fashion. }
\end{proof}

\subsection{A priori error analysis} \label{sec4:A priori}

We now embark on the error analysis of the DGFEM \eqref{galerkin_dg}. First, we point out that Galerkin orthogonality does not hold due to the inconsistency of $\tilde{B}(\cdot,\cdot)$. Thereby, we derive the following abstract error bound in the spirit of Strang's second lemma, cf. \cite[Theorem 4.2.2]{ciarlet}.
\begin{lemma} \label{sec4:strang}
Let $u\in H^2(\Omega)$ be the weak solution of~\eqref{pde}, \eqref{bcs}  
and $u_h\in \fes$ the DGFEM solution {defined} by~\eqref{galerkin_dg}.
Under the hypotheses of Lemma~\ref{sec4:lem:coercivity}, the following
abstract error bound holds
\begin{equation}\label{sec4:strang_inequality}
\ndg{u-u_h} \leq \left(1+\frac{C_{\rm cont}}{C_{\rm coer}} \right) \inf_{v_h\in \fes} \ndg{u-v_h}
+ \frac{1}{C_{\rm coer}}  \sup_{w_h\in\fes\backslash\{0\}} \frac{|\tilde{B} (u, w_h)-\tilde{\ell} (w_h)|}{\ndg{w_h}}.
\end{equation} 
\end{lemma}
We now derive the main results of this work.

\begin{theorem} \label{sec4:thm:apriori}
Let $\mesh=\{\k\}$ be a subdivision of $\Omega\subset\mathbb{R}^d$, $d=2,3$, 
consisting of general polytopic elements satisfying 
Assumptions~\ref{sec4:assumption_no_element_faces},~\ref{sec4:assumption_overlap} and \ref{sec4:assumption_shape_regular_mesh}, 
with $\coveringmesh=\{\mathcal{K}\}$ an associated covering of $\mesh$ consisting of shape-regular $d$--simplices, cf. Definition~\ref{sec4:definition_mesh_covering}.
Let $u_h\in \fes$, with $p_\k \geq2$ for all $\k\in\mesh$, be the corresponding DGFEM solution
defined by \eqref{galerkin_dg}, where the discontinuity-penalization function $\sigma$ and $\tau$ are given by~\eqref{sec4:eq:sigma}  and~\eqref{sec4:eq:tau}, respectively.
If the analytical solution $u\in H^2(\Omega)$  satisfies $u|_\k \in H^{l_\k}(\k)$, $l_\k>3+d/2$, for each $\k \in \mesh$, such that 
$\frak{E}u|_{\mathcal{K}}\in H^{l_\k}(\mathcal{K})$, where $\mathcal{K}\in\coveringmesh$ 
with $\k\subset\mathcal{K}$, then
\begin{equation}\label{sec4:final-error}
\ndg{u-u_{h}}^2 \le C\su \frac{h_{\k}^{2(s_{\k}-2)}}{p_{\k}^{2(l_{\k}-2)}} 
\left({\cal G}_\k(F, C_m,p_\k)
+{\cal D}_\k(F,C_{\rm INV},C_m,p_\k)
\right)
\|{\frak E} u\|_{H^{l_{\k}}(\mathcal{K})}^2,
\end{equation}
with $s_{\k}=\min\{p_{\k}+1,l_\k\}$, 
\begin{equation}
\begin{aligned} \label{sec4:con-error}
{\cal G}_\k(F,C_m,p_\k) 
:=
1 + \frac{h_\k^{-d+4}}{p_\k^{3} }
\sum_{F\subset\partial\k}
C_m(p_\k,\k,F)
\sigma|F| 
 +\frac{h_\k^{-d+2}}{p_\k}
\sum_{F\subset\partial\k}
C_m(p_\k,\k,F)
\tau|F|.  
\end{aligned}
\end{equation}
and 
\begin{equation}
\begin{aligned} \label{sec4:incon-error}
{\cal D}_\k(F,C_{\rm INV},C_m,p_\k) 
&:= \frac{h_\k^{-d-2}}{p_\k^{-3}} 
 \sum_{F\subset\partial\k}
C_m(p_\k,\k,F)
\sigma^{-1}|F| 
+
\frac{|\k|^{-1} h_{\k}^{-2}}{p^{-6}_{\k}}
\sum_{F\subset\partial\k}
C_{\rm INV}(p_{\k},\k,F)
\sigma^{-1}|F|
\\
& \hspace{-0cm}+ \frac{h_\k^{-d}}{p_\k^{-1}} 
 \sum_{F\subset\partial\k}
C_m(p_\k,\k,F)
\tau^{-1}|F| 
+
\frac{|\k|^{-1}}{p^{-2}_{\k}}
\sum_{F\subset\partial\k}
C_{\rm INV}(p_{\k},\k,F)
\tau^{-1}|F|
\end{aligned}
\end{equation}
where $C$ is a positive constant, which depends on the shape-regularity
of $\coveringmesh$, but is independent of the 
discretization parameters.
\end{theorem}
\begin{proof}
We start with the abstract bound \eqref{sec4:strang_inequality} in Lemma \ref{sec4:strang}. To bound the first term in the right-hand side of \eqref{sec4:strang_inequality}, we employ the approximation results in Lemma \ref{sec4:lemma_approx_polytopes}. We deduce the following bound 
\begin{equation} \label{sec4:dgnorm_approx}
\begin{aligned}
&\inf_{v_h\in \fes} \ndg{u-v_h}^2 
\leq \ndg{u-\tilde{\Pi}_{\bf p} u}^2  \\
&\quad  \leq \sum_{\el} \left(\ltwo{ \Delta (u-\tilde{\Pi}_{p_{\k}} u)}{L_2(\k)} ^2 
+2\sum_{F\subset\partial\k}\sigma  \|(u-\tilde{\Pi}_{p_{\k}} u)|_\k\|_{L_2(F)}^2
+2\sum_{F\subset\partial\k}\tau  \| \nabla(u-\tilde{\Pi}_{p_{\k}} u)|_\k\|_{L_2(F)}^2
 \right) \\
& \quad  \leq C\su \frac{h_{\k}^{2(s_{\k}-2)}}{p_{\k}^{2(l_{\k}-2)}} 
\left(1 + \frac{h_\k^{-d+4}}{p_\k^3} 
\sum_{F\subset\partial\k}
C_m(p_\k,\k,F)
\sigma|F| \right. \\
& \hspace{6cm}   \left. +\frac{h_\k^{-d+2}}{p_\k} 
\sum_{F\subset\partial\k}
C_m(p_\k,\k,F)
\tau|F|
\right) 
\|\frak{E} u\|_{H^{l_{\k}}(\mathcal{K})}^2,~~~~~~
\end{aligned}
\end{equation}
with $s_{\k}=\min\{p_\k+1, l_{\k}\}$ and $C$ a positive constant,
which depends on the shape-regularity of the covering $\coveringmesh$,
but is independent of the discretization parameters. 

We now proceed to bound the residual term arising in \eqref{sec4:strang_inequality}. After integration by parts, and noting that
$u$ is the solution of~\eqref{pde}, we get
\begin{equation} \label{sec4:incon-general}
\begin{aligned}
&\left|\tilde{B}_{\rm d} (u, w_h)-\tilde{\ell} (u, w_h)\right|
=
\Big|\int_{\Gamma}\mean{\nabla \Delta u- \nabla \ip(\Delta u)}\cdot\jump{w_h}
- 
\mean{ \Delta u-\ip (\Delta u)}\jump{\nabla  w_h}\ud s\Big| \\
& \quad \le
\Big(\int_{\Gamma} 
\sigma^{-1}|\mean{\nabla \Delta u- \nabla \ip(\Delta u)}|^2\ud s
+\tau^{-1}|\mean{ \Delta u- \ip(\Delta u)}|^2\ud s\Big)^{1/2}\ndg{w_h}.
\end{aligned}
\end{equation}
Next, we derive the error bound for the first term  in the {brackets of  \eqref{sec4:incon-general}. Adding and subtracting   $\nabla \Delta \tilde{\Pi}_{{\bf{p}}}u$,  
\begin{equation*}
\begin{aligned}
&\int_{\Gamma} 
\sigma^{-1}|\mean{\nabla \Delta u- \nabla \ip(\Delta u)}|^2\ud s \\
& ~\leq
\int_{\Gamma} 
2 \sigma^{-1}|\mean{\nabla  \Delta u-  \nabla   \Delta \tilde{\Pi}_{{\bf{p}}}u  }|^2\ud s
+2 \sigma^{-1}|\mean{\nabla \ip \Big(  \Delta \tilde{\Pi}_{{\bf{p}}} u - \Delta u \Big) }|^2\ud s 
\equiv \mbox{I}_1+\mbox{I}_2.
\end{aligned}
\end{equation*}}
Using, the above approximation result {\eqref{sec4:approx-inf_k_grad}}, we have 
\begin{equation} \label{in-relation1}
\mbox{I}_1\leq C\su  \frac{h_{\k}^{2(s_{\k}-3)}}{p_{\k}^{2(l_{\k}-3)}} 
\frac{h_\k^{-d}}{p_\k^{-1}} 
\left( \sum_{F\subset\partial\k}
C_m(p_\k,\k,F)
\sigma^{-1}|F|\right)
\|\frak{E} u\|_{H^{l_{\k}}(\mathcal{K})}^2.
\end{equation}
Similarly, employing the inverse inequalities~\eqref{sec4:inv_est_gen} and~\eqref{sec4:inverse_h1}, 
the $L_2$-stability of the projector $\ip$, and 
the approximation estimate~\eqref{sec4:approxH_k}, gives
\begin{equation} \label{in-relation2}
\mbox{I}_2 
\leq C\su  \frac{h_{\k}^{2(s_{\k}-2)}}{p_{\k}^{2(l_{\k}-2)}} 
\frac{|\k|^{-1}}{p_{\k}^{-2}}
 \frac{h_{\k}^{-2}}{p^{-4}_{\k}}
\left(\sum_{F\subset\partial\k}
C_{\rm INV}(p_{\k},\k,F)
\sigma^{-1}|F|
\right)
\|\frak{E} u\|_{H^{l_{\k}}(\mathcal{K})}^2.
\end{equation}
Next, we will use the same technique to bound the second term in the bracket of  \eqref{sec4:incon-general}. It is easy to see that following relation holds {
\begin{equation*}
\begin{aligned}
&\int_{\Gamma} 
\tau^{-1}|\mean{ \Delta u- \ip(\Delta u)}|^2\ud s \\
& ~\leq
\int_{\Gamma} 
2 \tau^{-1}|\mean{  \Delta u-     \Delta \tilde{\Pi}_{{\bf{p}}}u  }|^2\ud s
+2 \tau^{-1}|\mean{\ip \Big(  \Delta \tilde{\Pi}_{{\bf{p}}} u - \Delta u \Big) }|^2\ud s 
\equiv
\mbox{I}_3+\mbox{I}_4.
\end{aligned}
\end{equation*}}
with 
\begin{equation} \label{in-relation3}
\mbox{I}_3\leq C\su  \frac{h_{\k}^{2(s_{\k}-2)}}{p_{\k}^{2(l_{\k}-2)}} 
\frac{h_\k^{-d}}{p_\k^{-1}} 
\left( \sum_{F\subset\partial\k}
C_m(p_\k,\k,F)
\tau^{-1}|F|\right)
\|\frak{E} u\|_{H^{l_{\k}}(\mathcal{K})}^2.
\end{equation}
and 
\begin{equation} \label{in-relation4}
\mbox{I}_4 
\leq C\su  \frac{h_{\k}^{2(s_{\k}-2)}}{p_{\k}^{2(l_{\k}-2)}} 
\frac{|\k|^{-1}}{p_{\k}^{-2}}
\left(\sum_{F\subset\partial\k}
C_{\rm INV}(p_{\k},\k,F)
\tau^{-1}|F|
\right)
\|\frak{E} u\|_{H^{l_{\k}}(\mathcal{K})}^2.
\end{equation}
By inserting relation \eqref{in-relation1}, \eqref{in-relation2}, \eqref{in-relation3} and \eqref{in-relation4} into the bound \eqref{sec4:incon-general}, together with  relation \eqref{sec4:dgnorm_approx},  the final error bound \eqref{sec4:final-error} is derived.
\end{proof}
\begin{remark}
The above result generalizes well-known {\em a priori} bounds for DGFEMs 
defined on standard element shapes, cf.~\cite{MR2298696,MR2295480,MR2520159}, in two key ways. 
Firstly, meshes comprising polytopic elements are admitted; secondly, elemental  faces are allowed to degenerate. For $d=3$, this also implies that positive measure 
interfaces may have degenerating (one--dimensional) edges.
Thereby, this freedom is relevant to standard (simplicial/hexahedral) meshes with 
hanging nodes in the sense that no condition is required on the location
of hanging nodes on the element {boundary}.
If, on the other hand, the diameter of the faces of each element $\k\in \mesh$ 
is of comparable size to the diameter of the corresponding element, 
for uniform orders $p_\k =p \geq2$,  $h=\max_{\k\in\mesh} h_\k$, $s_\k=s$,
$s=\min\{p+1,l\}$, $l>3+d/2$, then the bound 
of Theorem~\ref{sec4:thm:apriori} reduces to
\begin{equation*}
\ndg{u-u_{h}}  \leq C \frac{h^{s-2}}{p^{l-\frac{7}{2}}} \| u\|_{H^{l}(\Omega)}.
\end{equation*}
This coincides with the analogous result derived in \cite{MR2295480} for standard 
meshes consisting of simplices or tensor-product elements. It is easy to check that  the above {\em a priori} error bound is  optimal in $h$ and suboptimal in $p$ by  $3/2$ order, as expected.
\end{remark}

\section{Error Analysis II: Arbitrary Number of Element Faces}\label{unbounded_faces}

In this section, we pursue the error analysis on meshes 
which potentially violate Assumption~\ref{sec4:assumption_no_element_faces}
in the sense that the number of faces that the elements possess
may not be uniformly bounded under mesh refinement. We note that 
this may arise when sequences of coarser meshes are generated via element agglomeration of a given fine mesh, cf. \cite{hp_multigrid_polytopes_2017}.
To this end, following Definition~\ref{sec4:element_face_simplices}, we introduce the following assumption on the mesh $\mesh$.
\begin{assumption}[Arbitrary number of faces] \label{sec5:new_assumption_no_element_faces} 
For any $\k \in\mesh$, there exists a set of non-overlapping $d$-dimensional 
simplices $\{\k^F_\flat\}_{F\subset \partial \k} \subset \mathcal{F}_{\flat}^{\k}$ 
contained in $\k$, such that for all $F \subset \partial \k$, the following condition
holds 
\begin{equation}\label{sec5:shape_relation1}
h_\k \leq C_s \frac{d|\k_\flat^F|}{|F|},
\end{equation}
where $C_s$ is a positive constant, which is independent of the discretization parameters, 
the number of faces that the element possesses, and the measure of $F$. 
\end{assumption}

\begin{figure}[t]
\begin{center}
	\includegraphics[scale=0.35]{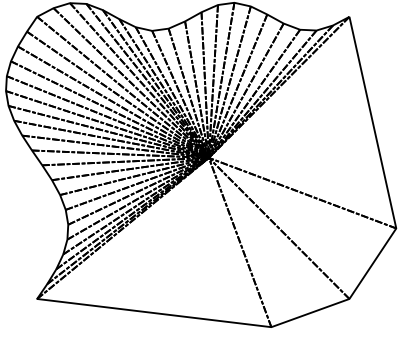}  \\
\end{center}
\caption{Polygon with many tiny faces.}
  \label{polygons}
\end{figure}
In Figure~\ref{polygons} we present one potential polygon in $\mathbb{R}^2$ which
satisfy the above mesh regularity assumption.
We note that Assumption~\ref{sec5:new_assumption_no_element_faces} 
does not place any restriction on either the number of faces that
an element $\k$, $\el$, may possess, or the relative
measure of its faces compared to the measure of the element itself. 
Indeed, shape-irregular simplices $\k_\flat^F$, with base $|F|$ of small 
size compared to the corresponding height, defined by $d|\k_\flat^F|/|F|$, are admitted.
However, the height must be of comparable size to $h_\k$, cf. 
the polygon depicted in Figure~\ref{polygons}. 
Furthermore, we note that the union of the simplices $\k_\flat^F$ does not need 
to cover the whole element $\k$, as in general it is sufficient to assume that
\begin{equation} \label{sec5:shape_relation2}
\bigcup_{F \subset \partial \k} \bar{\k}_\flat^F \subseteq \bar{\k}.
\end{equation}

\subsection{Inverse estimates} \label{sec5:inverse}
We will first revisit some {of} the  trace inverse estimates on simplices, cf.  \cite{warburton2003constants}.
\begin{lemma} \label{sec5:inverse_inequalties_standard_elements}
Given a simplex $T$ in $\mathbb{R}^d$, $d=2,3$, we write 
$F\subset\partial T$ to denote one of its faces. Then,
for $v\in \mathcal{P}_p(T)$, the following inverse inequality
holds
\begin{equation} \label{sec5:inv_est_classical} 
	\norm{v}{L_2(F)}^2 
	\le \frac{(p+1)(p+d)}{d}\frac{|F|}{|T|}\norm{v}{L_2(T)}^2.	
\end{equation}
\end{lemma}
By using Lemma \ref{sec5:inverse_inequalties_standard_elements}, together with the mesh Assumption~\ref{sec5:new_assumption_no_element_faces} on general polytopic meshes, the following  trace inverse inequality {holds (cf. \cite[Lemma 4.9]{Peter_phd}).}
\begin{lemma}\label{sec5:lemmal_inv_arbitrary}
Let $\el$; then assuming Assumption \ref{sec5:new_assumption_no_element_faces} 
is satisfied,
for each $v\in\mathcal{P}_p(\k)$, the following inverse inequality holds
\begin{equation}\label{sec5:inv_est_gen_arbitrary}
\norm{v}{L_2(\partial \k)}^2 \le C_s   \frac{{(p+1)(p+d)}}{h_\k} \norm{v}{L_2(\k)}^2. 
\end{equation}
The constant $C_s$ is defined in \eqref{sec5:shape_relation1}, and is
independent of $v$, $|\k|/\sup_{\k_{\flat}^F\subset \k}
|\k_{\flat}^F|$, $|F|$, and $p$.
\end{lemma}
Next, we introduce the \emph{harmonic polynomial space} on element $\k\in\mesh$, with polynomial degree $p$. 
\begin{equation}\label{sec5:def-harmonic}
\mathcal{H}_p(\k): = \{ u \in\mathcal{P}_{p}(\k), \Delta u = 0\}.
\end{equation}
Here, we point out that harmonic polynomials are commonly used in the context of \emph{Treffz-method}, see \cite{MR3177863,MR1939620}. It is easy to see that for $p=0,1$, $\mathcal{H}_p$ basis is identical to the $\mathcal{P}_p$ basis.

We will derive a new $H^1$-seminorm to $L_2$-norm inverse inequality for all \emph{harmonic polynomials},  based on the inequality \eqref{sec5:inv_est_gen_arbitrary}. 
\begin{lemma}\label{sec5: H1-lemmal_inv_arbitrary}
Let $\el$; then {if} Assumption \ref{sec5:new_assumption_no_element_faces} 
is satisfied,
for each $v\in\mathcal{H}_p(\k)$, the following inverse inequality holds
\begin{equation}\label{sec5: H1-inv_est_gen_arbitrary}
\norm{\nabla v}{L_2( \k)}^2 \le \Big( C_s   \frac{{(p+1)(p+d)}}{h_\k}\Big)^2 \norm{v}{L_2(\k)}^2. 
\end{equation}
Here, $C_s$ is defined in \eqref{sec5:shape_relation1}, and is
independent of $v$, $|\k|/\sup_{\k_{\flat}^F\subset \k}
|\k_{\flat}^F|$, $|F|$, and $p$.
\end{lemma}
\begin{proof}
We start {by using the} integration by parts formula and the fact that $\Delta v =0$, for all $v\in \mathcal{H}_p(k)$. Then, use the arithmetic--geometric mean inequality,  the  Cauchy--Schwarz inequality and $L_2$-norm of $\mbf{ n}$ is one.
$$
\norm{\nabla v}{L_2( \k)}^2 = \int_\k \nabla v \cdot \nabla v \ud \mbf{x} =  \int_{\partial \k} ({\bf n} \cdot \nabla v) v \ud{s} 
\leq \frac{\epsilon}{2}\norm{\nabla v}{L_2( \partial \k)}^2
 + \frac{1}{2\epsilon} \norm{ v}{L_2( \partial \k)}^2. 
$$
Employing the trace inverse inequality \eqref{sec5:inv_est_gen_arbitrary}, we have 
$$
\norm{\nabla v}{L_2( \k)}^2 
\leq \frac{\epsilon}{2} C_s   \frac{{(p+1)(p+d)}}{h_\k}\norm{\nabla v}{L_2(\k)}^2
 + \frac{1}{2\epsilon}  C_s   \frac{{(p+1)(p+d)}}{h_\k}\norm{ v}{L_2(\k)}^2. 
$$
By choosing $\epsilon = (C_s {{(p+1)(p+d)}})^{-1}h_\k$,  we deduce the desired result.
\end{proof}
\begin{remark}
We point out that the constant in the $H^1$-seminorm to $L_2$-norm inverse inequality \eqref{sec5: H1-inv_est_gen_arbitrary} is square of the constant in  $L_2$-norm trace inverse inequality  \eqref{sec5:inv_est_gen_arbitrary}. This bound is sharp in both $h$ and $p$ with respect to the  Sobolev index.  One direct application of this inverse inequality is to improve the result in  the multi-grid algorithm for DGFEM introduced in \cite{hp_multigrid_polytopes_2017}.  It can be proved that for DGFEM with $\mathcal{P}_1$ basis, the upper bound on the maximum eigenvalue of {the} discrete DGFEM bilinear form is independent of the number of elemental faces. It means that the smoothing step of  multi-grid algorithm   does not depend on the number of elemental faces for meshes satisfying the Assumption \ref{sec5:new_assumption_no_element_faces}.  
\end{remark}

\subsection{The stability of DGFEM} \label{sec5:stability}

In the rest of this work, we will only consider the case $p_\k=2,3$, 
$\k \in \mesh$, for  the stability analysis and  a priori error analysis. The key reason is that   the stability analysis  under the Assumption \ref{sec5:new_assumption_no_element_faces} requires the polynomial $v\in\fes$ satisfying the condition $\Delta^2 v= 0$ (Biharmonic polynomial functions). Only for $p_\k=2, 3$, $\k \in \mesh$, we know that polynomial basis $\mathcal{P}_{p_\k}$ {satisfies} the above condition.   For the consistency reason, we will still keep $p$ explicitly in the rest of {the work}.

First, we introduce the following bounded variation assumption on $\fes$. 
\begin{assumption}\label{sec5:assu-bounded}
For any $\k\in\mesh$, there exits a constant $\theta>1$, independent of all the discretization parameters, such that for any pair of elements $\k_i$ and $\k_j$ in $\mesh$ sharing a common face, {the} following bound holds
\begin{equation}\label{sec5: bounded}
\theta^{-1} \leq \Big( \frac{(p_{\k_i}+1)(p_{\k_i}+d)}{h_{\k_i}}\Big)/ \Big( \frac{(p_{\k_j}+1)(p_{\k_j}+d)}{h_{\k_j}}\Big)  \leq \theta,
\end{equation}
\end{assumption}
\begin{remark}
We point out that the above assumption is imposing bounded variation on $h$ and $p$ simultaneously. It is more general than the usual assumption which imposes the bounded variation on $h$ and $p$ separately. e.g. For $\k_i$,   $\k_j$ satisfying Assumption \ref{sec5:assu-bounded}, it is possible to have $h_{\k_i}$  much larger than  $h_{\k_j}$,  if we choose $p_{\k_i}$  also larger than  $p_{\k_j}$ based on relation \eqref{sec5: bounded}. However, this situation is not allowed under the usual bounded variation assumption.
\end{remark}
Next, we prove the coercivity and continuity for the inconsistent bilinear form $\tilde{B}(\cdot,\cdot)$ with respect to the norm $\ndg{\cdot}$, is established by the following lemma.

\begin{lemma} \label{sec5:lem:coercivity}
Given that Assumption~\ref{sec5:new_assumption_no_element_faces}, \ref{sec5:assu-bounded} hold, $p_\k =2,3$ for $\k\in \mesh$, {let 
$\sigma: \Gamma\rightarrow \mathbb{R_+}$  and $\tau: \Gamma\rightarrow \mathbb{R_+}$} to be defined facewise: 
\begin{equation}\label{sec5:eq:sigma}
\sigma(\mbf{x}) :=\left\{
\begin{array}{ll}
C_{\sigma}{\displaystyle \mean{ \Big( \frac{ ({\bf{p}}+1)({\bf{p}}+d) }{\mbf{h}} \Big)^3 } }, & ~  \uu{x}\in F \in \Gamma_{\dint}, ~F\subset\partial\k_i\cap\partial\k_j, \\
C_{\sigma}{\displaystyle  { \Big( \frac{ ({{p}_\k}+1)({{p}_\k}+d)}{{h}_\k} \Big)^3 } }, & ~  \uu{x}\in F \in \Gamma_{\ddd}, ~F\subset\partial\k.
\end{array}
\right. 
\end{equation}
and 
\begin{equation}\label{sec5:eq:tau}
\tau(\mbf{x}) :=\left\{
\begin{array}{ll}
C_{\tau}{\displaystyle \mean{  \frac{({\bf{p}}+1)({\bf{p}}+d)}{\mbf{h}}  } }, & ~  \uu{x}\in F \in \Gamma_{\dint}, ~F\subset\partial\k_i\cap\partial\k_j, \\
C_{\tau}{\displaystyle  {  \frac{ ({{p}_\k}+1)({{p}_k}+d)}{{h}_\k} }  }, & ~  \uu{x}\in F \in \Gamma_{\ddd}, ~F\subset\partial\k.
\end{array}
\right.
\end{equation}
Where $C_\sigma$ and  $C_\tau$ are  sufficiently large positive
constants. Then the bilinear form $\tilde{B}(\cdot,\cdot)$ is coercive and continuous over $\mathcal{S}\times \mathcal{S}$, i.e.,
\begin{equation}\label{coer}
\tilde{B}(v,v) \geq C_{\rm {coer}} \ndg{v}^2 \quad\text{for all}\quad v\in \mathcal{S},
\end{equation}
and
\begin{equation}\label{cont}
\tilde{B}(w,v)\leq C_{\rm cont}\ndg{w}\,\ndg{v} \quad\text{for all}\quad w,v\in \mathcal{S},
\end{equation}
respectively,
where $C_{\rm coer}$ and $C_{\rm cont}$ are positive constants, independent of the local
mesh sizes $h_\k$, local polynomial degree orders ${p_\k}$, $\el$, measure of faces  and number of elemental faces.
\end{lemma}

\begin{proof}
Recalling the second term on the right-hand side of \eqref{sec4:coercivity-1} in the proof of Lemma \ref{sec4:lem:coercivity}, upon application of the  Cauchy--Schwarz inequality,  the arithmetic--geometric mean inequality, inverse inequality \eqref{sec5:inv_est_classical}, relation \eqref{sec5:shape_relation1}  and \eqref{sec5:shape_relation2}, {we} deduce that
\begin{equation*} 
\begin{aligned}
 \int_{\Gamma}\mean{\nabla \ip (\Delta v)} \cdot \jump{v} \ud s 
 &\leq 
\epsilon \su \sum_{F \subset \partial \k}\norm{{\sqrt{\sigma^{-1}}}  \nabla \ip(\Delta v)}{L_2(F)}^2    
 + \frac{1}{4\epsilon} \int_{\Gamma}{\sigma|\jump{v}|}^2 \ud{s} \\ 
 & \leq 
 \epsilon \su \sum_{F \subset \partial \k}
\sigma^{-1}  C_s   \frac{{(p_\k+1)(p_\k+d)}}{h_\k}
\norm{  \nabla \ip(\Delta v)}{L_2(\k^F_\flat)}^2    
 + \frac{1}{4\epsilon} \int_{\Gamma}{\sigma|\jump{v}|}^2 \ud{s} \\
& \leq 
 \epsilon \su (\max_{F \subset \partial \k}
\sigma^{-1})  C_s   \frac{{(p_\k+1)(p_\k+d)}}{h_\k}
\norm{  \nabla \ip(\Delta v)}{L_2(\k)}^2    
 + \frac{1}{4\epsilon} \int_{\Gamma}{\sigma|\jump{v}|}^2 \ud{s} .
\end{aligned}
\end{equation*}
Using the inverse inequality \eqref{sec5: H1-inv_est_gen_arbitrary}
 stated in Lemma~\ref{sec5: H1-lemmal_inv_arbitrary}, Assumption \ref{sec5:assu-bounded},  the stability of the $L_2$-projector $\ip$ in the $L_2$-norm together with the Definition of $\sigma$ in { \eqref{sec5:eq:sigma}} we deduce
\begin{equation}  \label{sec5:coer-relation1}
\begin{aligned}
 \int_{\Gamma}\mean{\nabla \ip (\Delta v)} \cdot \jump{v} \ud s  
&\leq
 \epsilon \su (\max_{F \subset \partial \k}
\sigma^{-1})  \Big( C_s   \frac{{(p_\k+1)(p_\k+d)}}{h_\k} \Big)^3
\norm{ \ip(\Delta v)}{L_2(\k)}^2    
 + \frac{1}{4\epsilon} \int_{\Gamma}{\sigma|\jump{v}|}^2 \ud{s}  \\
& \leq 
   \frac{\epsilon (C_s \theta)^3 }{C_\sigma}
\su \norm{ \Delta v}{L_2(\k)}^2    
 + \frac{1}{4\epsilon} \int_{\Gamma}{\sigma|\jump{v}|}^2 \ud{s}.
  \end{aligned}
\end{equation}
Next, we  bound the last term on the right-hand side of \eqref{sec4:coercivity-1}. We point out that only using the trace inverse inequality \eqref{sec5:inv_est_classical} in Lemma~\ref{sec5:inverse_inequalties_standard_elements} and Definition of $\tau$ in \eqref{sec5:eq:tau}, the following relation {holds:  }
\begin{equation} \label{sec5:coer-relation2}
\begin{aligned}
 \int_{\Gamma}\mean{ \ip (\Delta v)}  \jump{\nabla v} \ud s  
 & \leq \frac{\epsilon C_s \theta}{C_\tau}  \su \norm{ \Delta v}{L_2(\k)}^2    
 + \frac{1}{4\epsilon} \int_{\Gamma}{\tau|\jump{\nabla v}|}^2 \ud{s}. 
  \end{aligned}
\end{equation}
Inserting { \eqref{sec5:coer-relation1} and \eqref{sec5:coer-relation2} into \eqref{sec4:coercivity-1}}, we deduce that 
$$
\tilde{B}(v,v)
\geq 
\Big(1 - \frac{2\epsilon (C_s \theta)^3 }{C_\sigma}  -  \frac{2\epsilon C_s \theta}{C_\tau}
\Big) 
\sum_{\el} \ltwo{\Delta v}{L_2(\k)} ^2 
+
\Big(1 -\frac{1}{2\epsilon} \Big) 
\int_{\Gamma}  \Big( \sigma |\jump{v}|^2 
+
\tau |\jump{\nabla v}|^2 \Big)\ud{s},
$$
because the constant $C_s$ and $\theta$ are bounded by Assumption \ref{sec5:new_assumption_no_element_faces} and \ref{sec5:assu-bounded}, respectively. So the bilinear form $\tilde{B}(\cdot,\cdot)$ is coercive over $\mathcal{S} \times \mathcal{S}$, if $C_\sigma > 4\epsilon (C_s \theta)^3$, $C_\tau > 4\epsilon C_s \theta$ and $\epsilon>\frac{1}{2}$.  The proof of continuity follows immediately.  
\end{proof}
\begin{remark}
We point out that it is possible to prove  Lemma \ref{sec5:lem:coercivity} without using  Assumption \ref{sec5:assu-bounded}. The technical difficulty is that in order to use the inverse inequality \eqref{sec5: H1-inv_est_gen_arbitrary} stated in Lemma~\ref{sec5: H1-lemmal_inv_arbitrary},   it is essential to take out the face-wise penalization function $\sigma$ out of the summation over all the element boundary $\partial \k$, which means that the information of $\sigma$ on $F$, $F\subset \partial \k$, is coupled  over all the element boundary.  It implies that the definition of the function $\sigma$ on face $F$, shared by $\k_i$ and $\k_j$,  depends not only on the discretization parameters of $\k_i$, $\k_j$ but also on  all of their neighbouring elements.  This  definition {may be} not practical under the mesh Assumption~\ref{sec5:new_assumption_no_element_faces}, which {allows} the number of  neighbouring elements to be arbitrary.
\end{remark}

\subsection{A priori error analysis} \label{sec5:A priori}
Before deriving  the a priori error bound for the 
DGFEM~\eqref{galerkin_dg} under the  
Assumption~\ref{sec5:new_assumption_no_element_faces}, we recall  the approximation result for function $u$ on the element boundary.

\begin{lemma}\label{sec5:lemma_approx_polytopes_arbitrary}
Let $\el$ and 
$\mathcal{K}\in\coveringmesh$ the corresponding simplex 
such that $\k\subset\mathcal{K}$, satisfying the Definition~\ref{sec4:definition_mesh_covering}.
Suppose that $v\in H^1(\Omega)$ is such that 
$\frak{E}v|_{\mathcal{K}}\in H^{l_{\k}}(\mathcal{K})$. Then, given that
Assumption \ref{sec5:new_assumption_no_element_faces} is satisfied, the
following bound holds {
\begin{equation}\label{sec5:approx-inf_k-arbitrary}
\norm{v - \tilde{\Pi}_p v}{L_2(\partial \k)}
\le C_{4}\frac{h_{\k}^{s_{\k}-1/2}}{p^{l_{\k}-1/2}} 
\norm{{\cal E} v}{H^{l_\k}(\mathcal{K})}, \quad l_\k\ge 1/2,
\end{equation}
and
\begin{equation}\label{sec5:approx-inf_k-arbitrary_grad}
\norm{\nabla(v - \tilde{\Pi}_p v)}{L_2(\partial \k)}
\le C_{5}\frac{h_{\k}^{s_{\k}-3/2}}{p^{l_{\k}-3/2}} 
\norm{{\cal E} v}{H^{l_\k}(\mathcal{K})}, \quad l_\k\ge 3/2,
\end{equation}}
where
$s_{\k}=\min\{p+1, l_{\k}\}$, $C_{4}$ and $C_{5}$
are  positive constants depending on $C_s$ from \eqref{sec5:shape_relation1} and 
the shape-regularity of ${\mathcal{K}}$, but is independent of $v$, $h_{\k}$, $p$, and number of faces per element.
\end{lemma}
\begin{proof}
{
The proof for  relation \eqref{sec5:approx-inf_k-arbitrary}  can be found  in \cite[Lemma $33$]{DGpolybook}.  The relation \eqref{sec5:approx-inf_k-arbitrary_grad} can be proven in the similar fashion. }
\end{proof}

Next, we derive the a priori error bound with help of above the Lemma.

\begin{theorem} \label{sec5:thm:apriori}
Let $\mesh=\{\k\}$ be a subdivision of $\Omega\subset\mathbb{R}^d$, $d=2,3$, 
consisting of general polytopic elements satisfying 
Assumptions~\ref{sec4:assumption_overlap},  \ref{sec5:new_assumption_no_element_faces} and \ref{sec5:assu-bounded}, 
with $\coveringmesh=\{\mathcal{K}\}$ an associated covering of $\mesh$ consisting of shape-regular $d$--simplices, cf. Definition~\ref{sec4:definition_mesh_covering}.
Let $u_h\in \fes$, with $p_\k  = 2, 3$ for all $\k\in\mesh$, be the corresponding DGFEM solution
defined by \eqref{galerkin_dg}, where the discontinuity-penalization function $\sigma$ and $\tau$ are given by~\eqref{sec5:eq:sigma}  and~\eqref{sec5:eq:tau}, respectively.
If the analytical solution $u\in H^2(\Omega)$  satisfies $u|_\k \in H^{l_\k}(\k)$, $l_\k>7/2$, for each $\k \in \mesh$, such that 
$\frak{E}u|_{\mathcal{K}}\in H^{l_\k}(\mathcal{K})$, where $\mathcal{K}\in\coveringmesh$ 
with $\k\subset\mathcal{K}$, then
\begin{equation}\label{sec5:final-error}
\ndg{u-u_{h}}^2 \le C\su \frac{h_{\k}^{2(s_{\k}-2)}}{p_{\k}^{2(l_{\k}-2)}} 
\left(
{\cal G}_\k(h_\k,p_\k)
+{\cal D}_\k(h_\k,p_\k)
\right)
\|{\frak E} u\|_{H^{l_{\k}}(\mathcal{K})}^2,
\end{equation}
with $s_{\k}=\min\{p_{\k}+1,l_\k\}$, 
\begin{equation}
\begin{aligned} \label{sec5:con-error}
{\cal G}_\k(h_\k,p_\k) 
:=
1 + \frac{h_\k^{3}}{p_\k^{3} }
\Big( \max_{F\subset \partial \k  }\sigma|_F \Big)
 +\frac{h_\k}{p_\k}
\Big( \max_{F\subset \partial \k  }\tau|_F \Big).  
\end{aligned}
\end{equation}
and 
\begin{equation}
\begin{aligned} \label{sec5:incon-error}
{\cal D}_\k(h_\k,p_\k) 
&:= \frac{h_\k^{-3}}{p_\k^{-3}} 
\Big( \max_{F\subset \partial \k  }\sigma^{-1}|_F \Big)
+
\frac{ h_{\k}^{-3}}{p^{-6}_{\k}}
\Big( \max_{F\subset \partial \k  }\sigma^{-1}|_F \Big)
\\
& \hspace{-0cm}+ \frac{h_\k^{-1}}{p_\k^{-1}} 
\Big( \max_{F\subset \partial \k  }\tau^{-1}|_F \Big)
+
\frac{h^{-1}}{p^{-2}_{\k}}
\Big( \max_{F\subset \partial \k  }\tau^{-1}|_F \Big)
\end{aligned}
\end{equation}
where $C$ is a positive constant, which depends on the shape-regularity
of $\coveringmesh$ and $C_s$, but is independent of the 
discretization parameters and number of elemental faces.
\end{theorem}
\begin{proof}
The proof of the error bound follows the same way of proof for Theorem \ref{sec4:thm:apriori}.  For brevity, we  focus on the terms defined on the faces of the elements in the computational mesh only. Thereby,
employing \eqref{sec5:approx-inf_k-arbitrary} in 
Lemma~\ref{sec5:lemma_approx_polytopes_arbitrary}, we deduce that
\begin{equation} 
\begin{aligned}
\int_{\Gamma}\sigma \jump{v - \tilde{\Pi}_p v}^2\ud s  
&\leq 2 \su \sum_{F\subset \partial \k } \sigma|_F \norm{v - \tilde{\Pi}_p v}{L_2(F)}^2  \\
&\leq 2 \su \Big(\max_{F\subset \partial \k}\sigma|_F \Big) \norm{v - \tilde{\Pi}_p v}{L_2(\partial \k)}^2 
  \leq C \su \Big( \max_{F\subset \partial \k  }\sigma|_F \Big) \frac{h_{\k}^{2s_{\k}-1}}{p^{2l_{\k}-1}} 
\norm{{\cal E} v}{H^{l_\k}(\mathcal{K})}^2.
\end{aligned}
\end{equation}
 With help of Lemma \ref{sec5:lemma_approx_polytopes_arbitrary}, \ref{sec5:inverse_inequalties_standard_elements} and  \ref{sec5: H1-lemmal_inv_arbitrary}, the inconsistent error is derived in the similar way,
\end{proof}

\begin{remark}
{We point out that  the regularity requirements for trace approximation result in Lemma \ref{sec5:lemma_approx_polytopes_arbitrary} is lower compared to the trace approximation result  in Lemma \ref{sec4:lemma_approx_polytopes}, which requires the control of the $L_\infty$-norm of the function.}  Similarly, by using the Assumption \ref{sec5:assu-bounded}, 
for uniform orders $p_\k =2,3$,  $h=\max_{\k\in\mesh} h_\k$, $s_\k=s$,
$s=\min\{p+1,l\}$, $l>7/2$, then the bound 
of Theorem~\ref{sec5:thm:apriori} reduces to
\begin{equation*}
\ndg{u-u_{h}}  \leq C \frac{h^{s-2}}{p^{l-\frac{7}{2}}} \| u\|_{H^{l}(\Omega)}.
\end{equation*}
The constant $C$ depends on the constant $C_s$ and $\theta$, but independent of the measure of the elemental faces and the number of elemental faces.
\end{remark}

\section{Numerical Examples}\label{numerical example}

In this section, we present a series of numerical examples to illustrate the a priori error estimates derived in this work. { Throughout this section the DGFEM solution $u_h$ defined by \eqref{galerkin_dg} is computed.  The constants $C_\sigma$ and $C_\tau $, appearing in the discontinuity penalization functions  $\sigma$ and $\tau$,  respectively, are both equal to $10$.}  

\subsection{Example 1}
In this example, let $\Omega:=(0,1)^2$  and select $f$, $g_{\rm D}$ and $g_{\rm N}$ such that the analytical solution to \eqref{pde} and \eqref{bcs} is give as 
$$
u(x,y) = \sin(\pi x)^2  \sin(\pi y)^2.
$$
The polygonal meshes used in this example are generated through the PolyMesher MATLAB library \cite{polymesher}. In general, these polygonal meshes contain no more than $8$ edges, which {satisfy}  Assumption \ref{sec4:assumption_no_element_faces}.  Typical meshes generated by PolyMesher are shown in Figure \ref{ex1:mesh_figure}.

\begin{figure}[!h]
\centering
\subcaptionbox{$\mathcal T$ with $64$
polygons.}{
\includegraphics[scale=0.3]{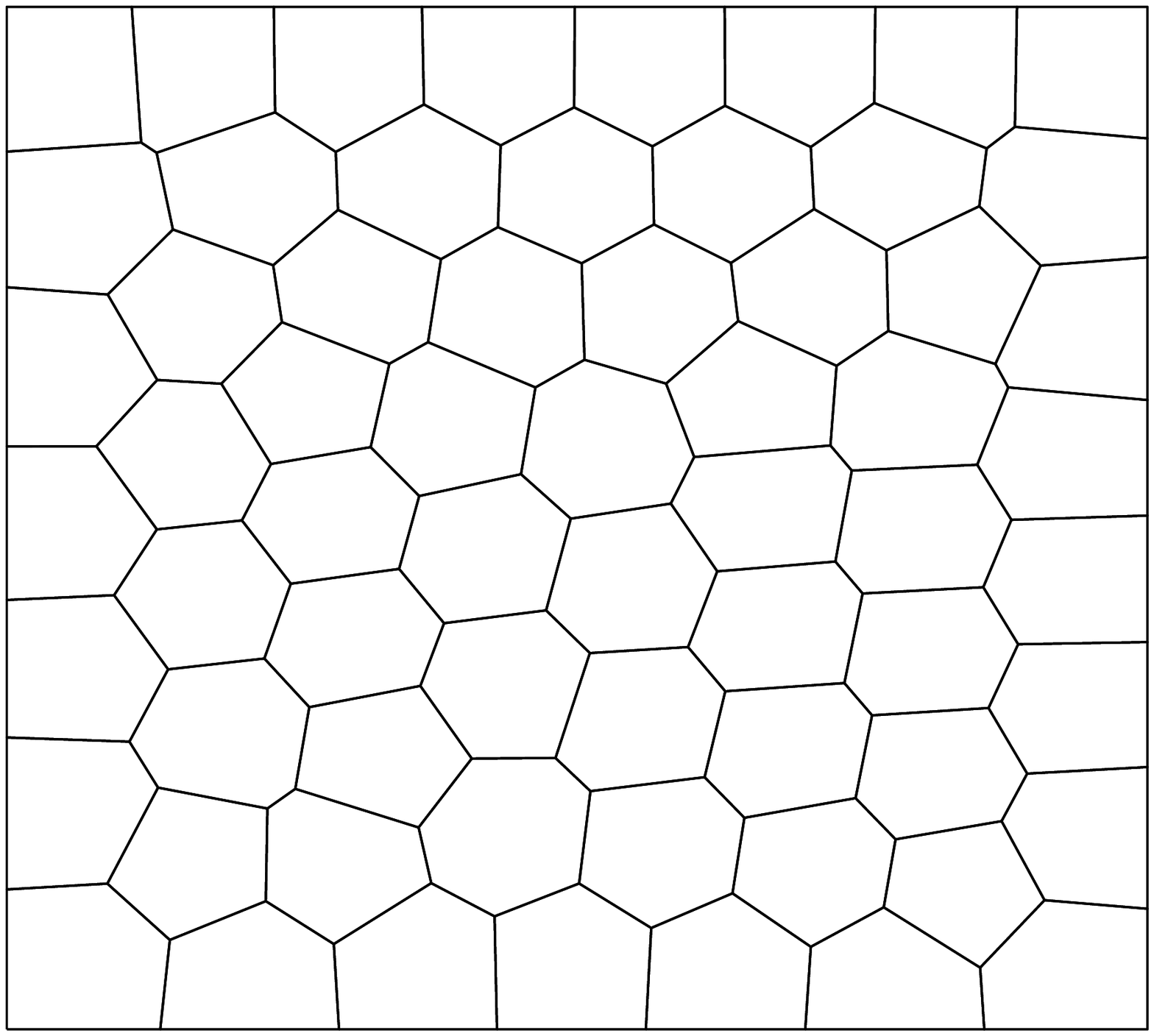}
}
\hspace{0.5cm}
\subcaptionbox{$\mathcal T$ with $256$
polygons.}{
\includegraphics[scale=0.3]{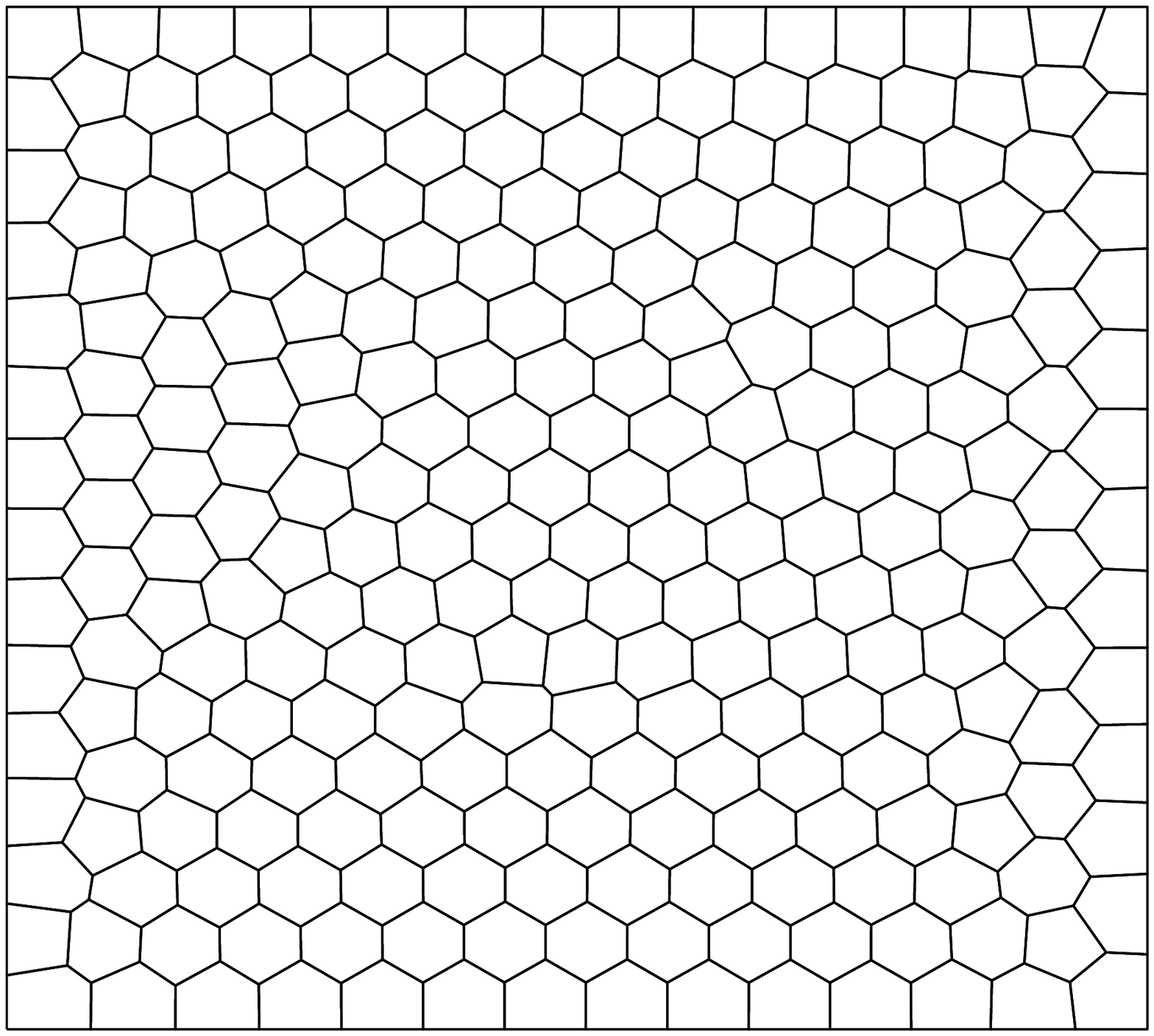} 
}
\caption{
\label{ex1:mesh_figure}
An example of a polytopic mesh $\mathcal T$ with bounded number of elemental faces.}
\end{figure}

\begin{figure}[!ht]
\centering
\includegraphics[scale=0.38]{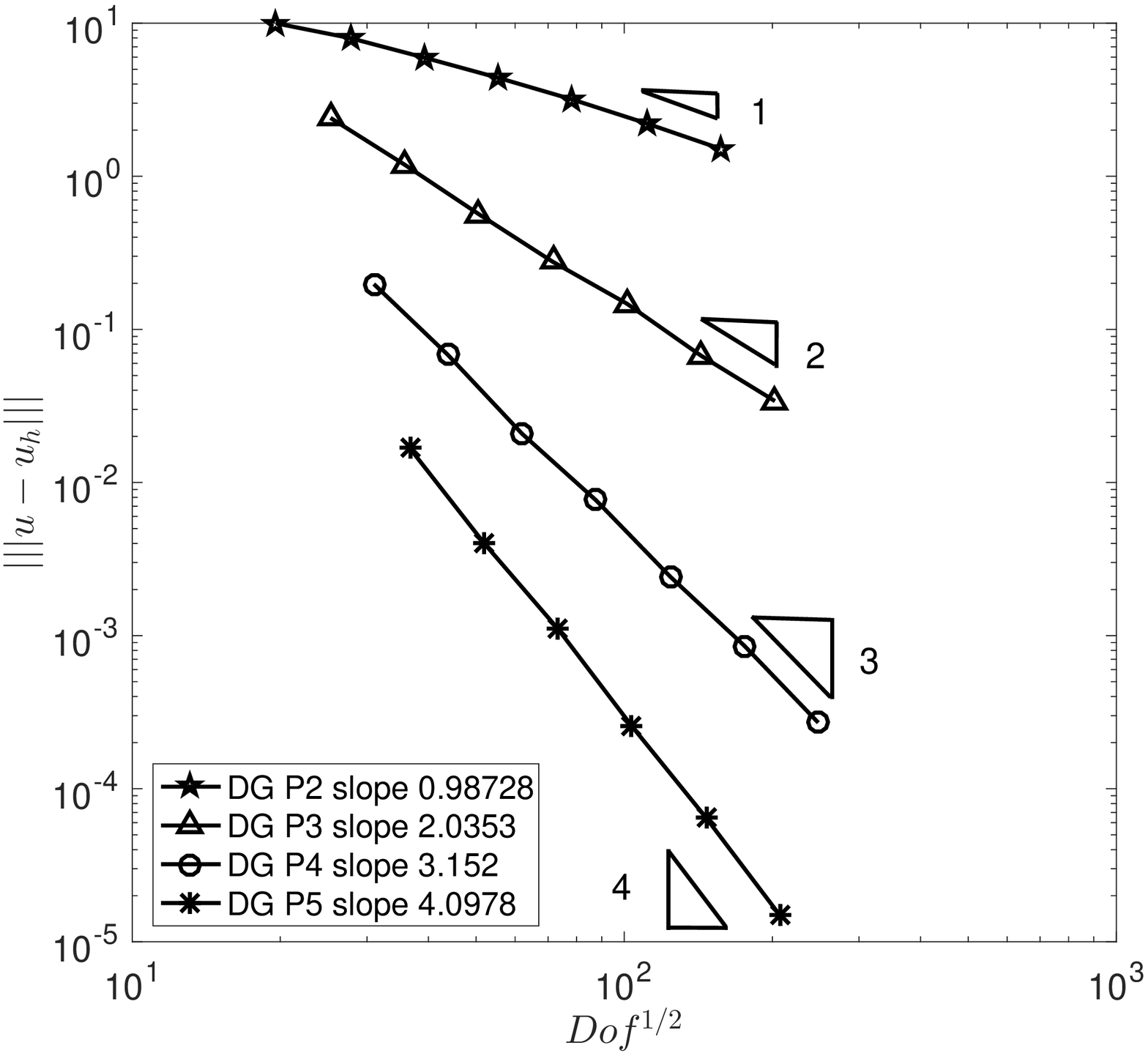} 
\hspace{0cm}
\includegraphics[scale=0.38]{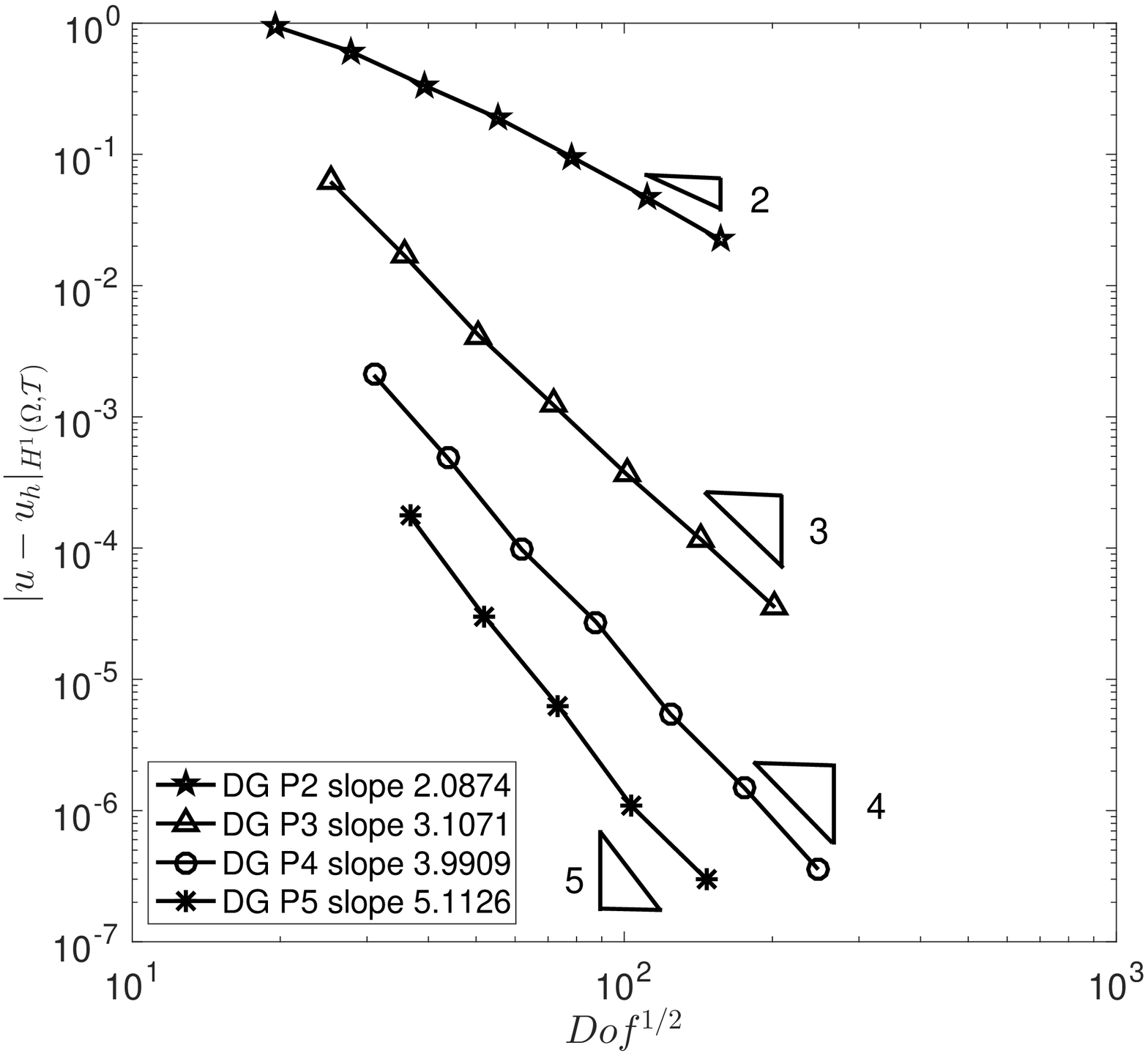} 
\\
\includegraphics[scale=0.38]{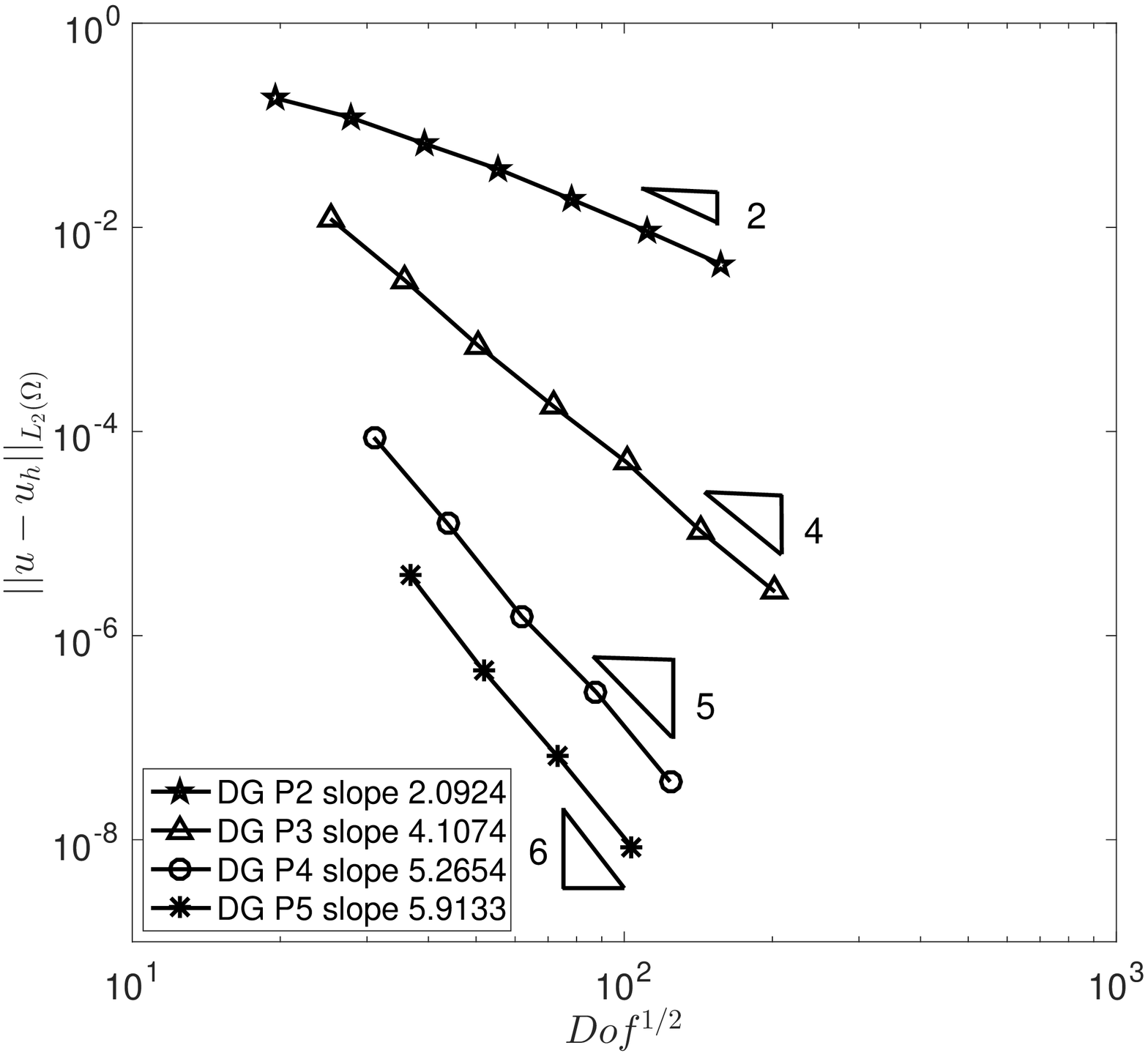} 
\caption{
\label{ex1:h-refine}
 Convergence of the DGFEM under
$h$--refinement for $p=2,3,4,5$. }
\end{figure}

We will investigate the asymptotic behaviour of the errors of the DGFEM on a sequence of finer polygonal meshes for different $p=2,3,4,5$. In Figure \ref{ex1:h-refine}, we present the  DG-norm, broken $H^1$-seminorm and  $L_2$-norm error in the approximation to $u$. First, we observe that $\ndg{u-u_h}$ converges to zero at the optimal rate $\mathcal{O}(h^{p-1})$ for fixed $p$ as the mesh becomes finer, which confirms the error bound in Theorem \ref{sec4:thm:apriori}. Second, we observe that the $|u-u_h|_{H^1(\Omega,\mesh)}$ converges to zero at the optimal rate $\mathcal{O}(h^{p})$ for fixed $p$. Finally, we observe that the $\norm{u-u_h}{L_2(\Omega)}$ converges to zero at the optimal rate $\mathcal{O}(h^{p+1})$ for fixed $p\geq3$. However, for $p=2$, the convergence rate is only $\mathcal{O}(h^{2})$. This sub-optimal convergence result has been observed by other researchers, see \cite{MR2298696,MR2520159}.

Finally, we will investigate the convergence behaviour of DGFEM under $p$-refinement on the fixed polygonal meshes. To this end, in Figure \ref{ex1:p-refine}, we plot the DG-norm error against polynomial degree $p$ in linear-log scale for four different polygonal meshes. For all cases, we observe that the convergence plot is straightly which shows the error decays exponentially.  

\begin{figure}[!ht]
\centering
\includegraphics[scale=0.4]{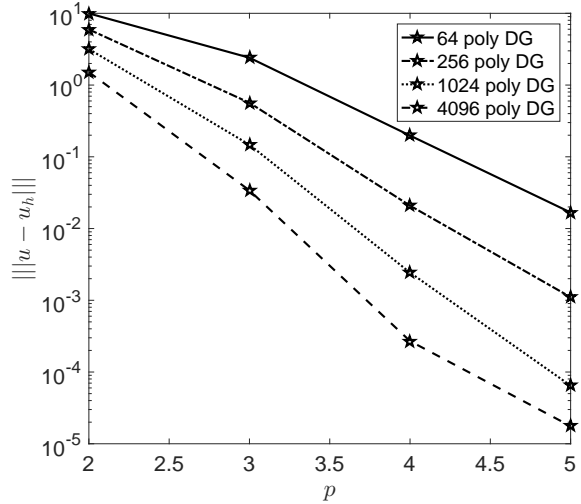} 
\caption{
\label{ex1:p-refine}
 Convergence of the DGFEM under
$p$--refinement. }
\end{figure}

\subsection{Example 2}

In this example, let $\Omega:=(0,1)^2$  and select $f$, $g_{\rm D}$ and $g_{\rm N}$ such that the analytical solution to \eqref{pde} and \eqref{bcs} is give as 
$$
u(x,y) = x(1-x)y(1-y).
$$
The DGFEM solution is computed on general polygonal meshes with a lot of tiny faces, stemming from the agglomeration of a given (fixed) fine mesh consisting of $524,288$ triangular elements. The mesh agglomeration procedure is done in a rough way such that the resulting polygonal meshes generated by this procedure will contain more than $500$ edges {at} the coarsest level. We emphasize that the reason for using the  polygonal meshes with a lot {of} faces is to investigate the stability and accuracy for the proposed DGFEM.  We will use polygonal meshes {consisting} of $32$, $134$, $512$, $2048$, and $8192$ elements in the computation. In Figure \ref{ex2:mesh_figure}, we show some of polygonal meshes used in this example.

\begin{figure}[!h]
\centering
\subcaptionbox{$\mathcal T$ with $32$
polygons.}{
\includegraphics[scale=0.35]{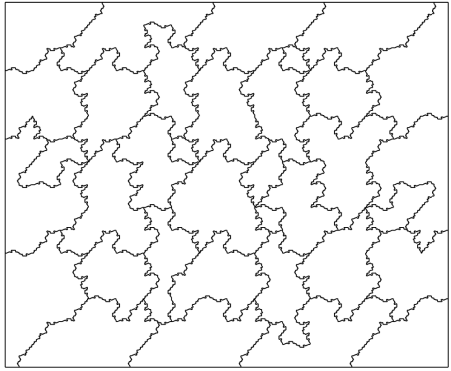}
}
\hspace{0.2cm}
\subcaptionbox{$\mathcal T$ with $512$
polygons.}{
\includegraphics[scale=0.35]{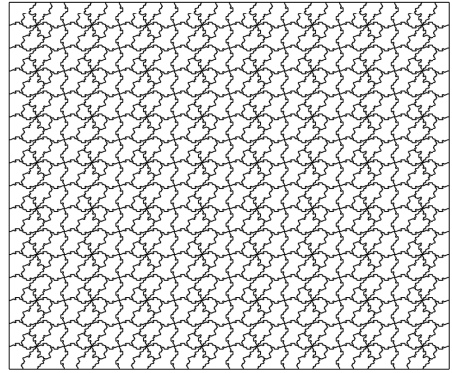} 
}
\caption{
\label{ex2:mesh_figure}
An example of a polytopic mesh $\mathcal T$ with a lot of  elemental faces.}
\end{figure}

\begin{figure}[!ht]
\centering
\includegraphics[scale=0.38]{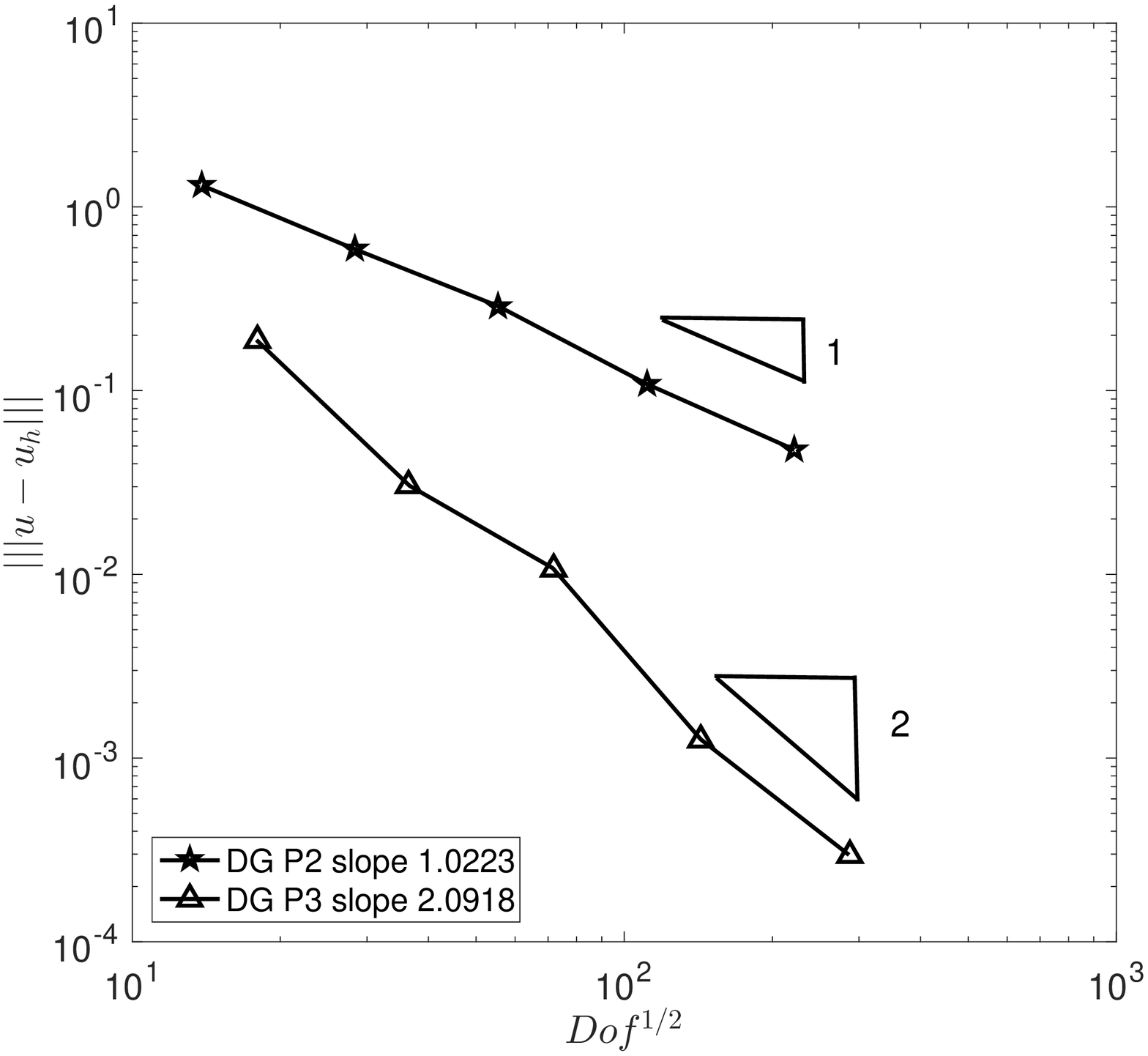} 
\hspace{0cm}
\includegraphics[scale=0.38]{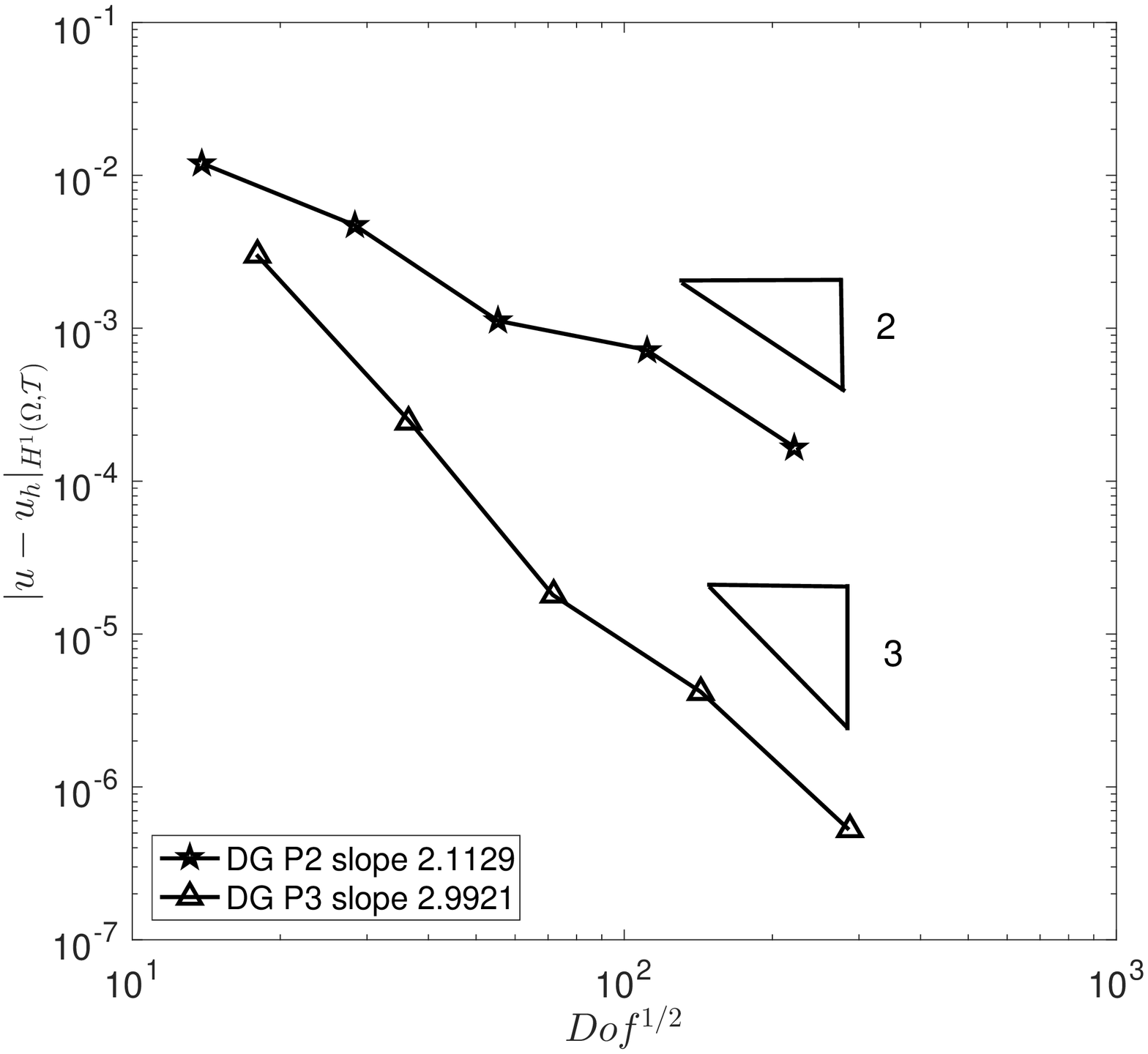} 
\\
\includegraphics[scale=0.38]{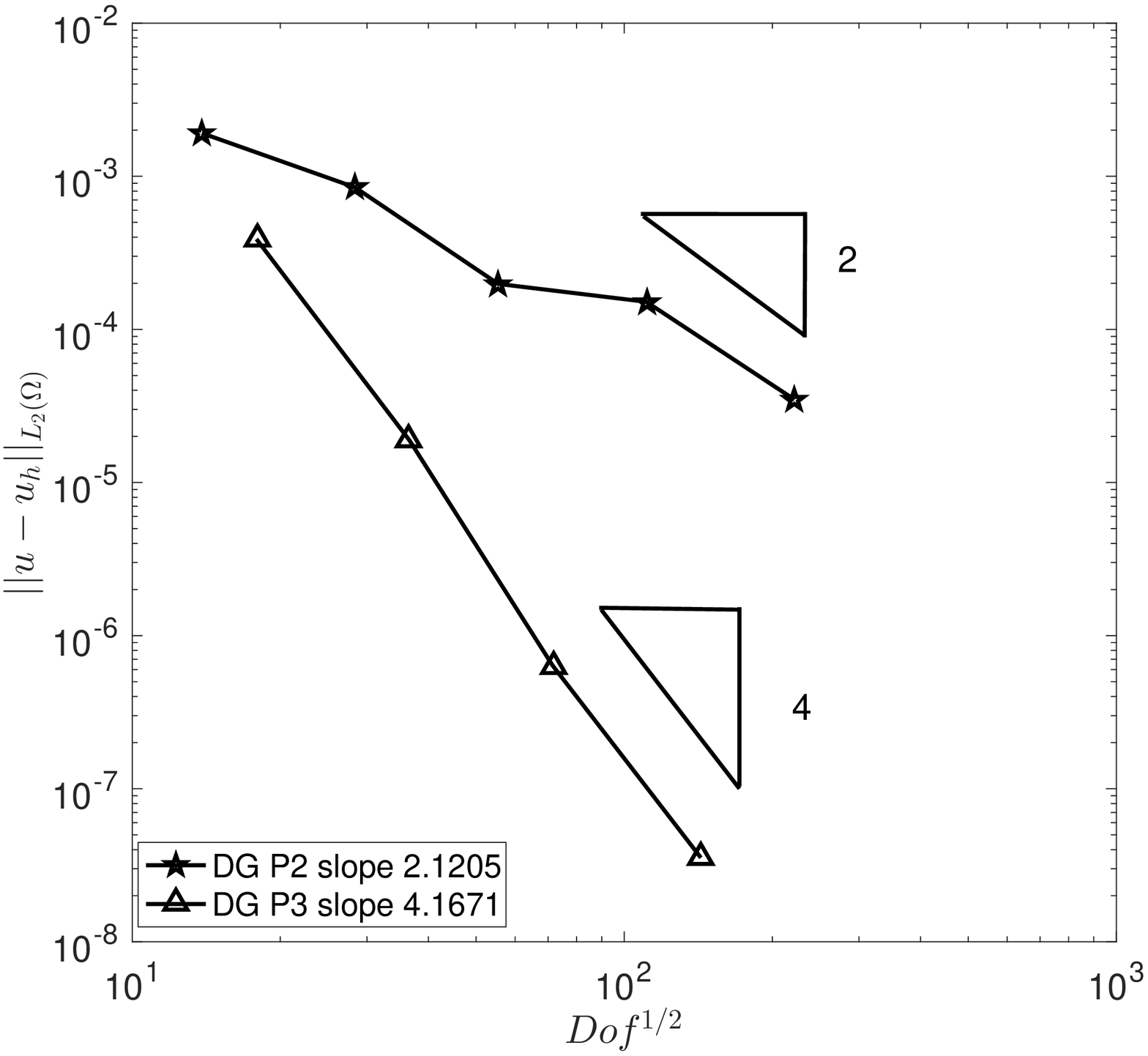} 
\caption{
\label{ex2:h-refine}
 Convergence of the DGFEM under
$h$--refinement for $p=2,3$. }
\end{figure}

We will investigate the asymptotic behaviour of the errors of the DGFEM on a sequence of finer polygonal meshes for different $p=2,3$. In Figure \ref{ex2:h-refine}, we present the  DG-norm, broken $H^1$-seminorm and $L_2$-norm error in the approximation to $u$. First, we observe that the $\ndg{u-u_h}$ converges to zero at the optimal rate $\mathcal{O}(h^{p-1})$ for fixed $p$ as the mesh becomes finer, which confirms the error bound in Theorem \ref{sec5:thm:apriori}. Second, we observe that  $|u-u_h|_{H^1(\Omega,\mesh)}$ converges to zero at the optimal rate $\mathcal{O}(h^{p})$ for fixed $p$. Third, we observe that  $\norm{u-u_h}{L_2(\Omega)}$ converges to zero at the optimal rate $\mathcal{O}(h^{p+1})$ for fixed $p=3$. For $p=2$, the convergence rate is only $\mathcal{O}(h^{2})$ as we expected. Finally, we mention that by choosing the discontinuity penalization functions  $\sigma$ and $\tau$ defined in  Lemma \ref{sec5:lem:coercivity}, there is no numerical instability observed in the computation. The condition number for the proposed DGFEM employing the polygonal meshes with a lot of tiny faces is at the same level of the condition number for the DGFEM employing the polygonal meshes with less than $10$ faces.

\section{Concluding Remarks} \label{conclusion}

We have studied the $hp$-version IP DGFEM for {the} biharmonic boundary value problem, based on employing general computational meshes containing polygonal/polyhedral elements with degenerating $(d-k)$-dimensional faces, $k=1,\dots,d-1$. The key results in this work are that {the} $hp$-version IP-DGFEM is stable on polygonal/polyhedral elements satisfying the assumption that the number of elemental faces is uniformly bounded. Moreover, with {the} help of the new inverse inequality in Lemma \ref{sec5: H1-lemmal_inv_arbitrary}, we also prove that IP-DGFEM employing $\mathcal{P}_p$ basis, $p=2,3$, is stable on polygonal/polyhedral elements with arbitrary number of faces satisfying Assumption \ref{sec5:new_assumption_no_element_faces}.  The numerical examples also confirm the theoretical analysis. 

From the practical point of view, the condition number for DGFEM to solve biharmonic problem is typically very large. The development of   efficient multi-grid solvers for the DGFEM on general polygonal/polyhedral meshes, is left as {a} further challenge.

\section*{Acknowledgements} We wish to express our sincere gratitude to Andrea Cangiani (University of Leicester) and Emmanuil Georgoulis  (University of Leicester \& National Technical University of Athens) for their insightful comments on an earlier version of this work. Z. Dong  acknowledges funding by The Leverhulme Trust (grant no. RPG-2015-306).

\bibliographystyle{siam}
\bibliography{DGpoly_Biharmonic_revision}

\end{document}